\documentclass[12pt]{article}

\usepackage[T1]{fontenc}
\usepackage{amsthm, amssymb}
\usepackage{mathtools}
\usepackage{mathabx}
\usepackage{dsfont}
\usepackage{enumitem}
\usepackage{authblk}
\usepackage{xcolor}
\mathtoolsset{showonlyrefs}

\newcommand{\vc}[1]{\boldsymbol{#1}}
\newcommand{\ind}{\mathds{1}}
\newcommand{\inner}[2]{\langle #1, #2\rangle}
\newcommand{\conj}[1]{\overline{#1}\,}
\newcommand{\R}{\mathbb{R}}
\newcommand{\Z}{\mathbb{Z}}

\DeclarePairedDelimiter\abs{\lvert}{\rvert}
\DeclarePairedDelimiter\norm{\lVert}{\rVert}
\DeclarePairedDelimiter\japan{\langle}{\rangle}
\DeclarePairedDelimiter\floor{\lfloor}{\rfloor}

\DeclareMathOperator{\supp}{supp}
\DeclareMathOperator{\real}{Re}

\DeclareMathOperator{\sgn}{sgn}
\DeclareMathOperator{\BigO}{\mathcal{O}}
\newcommand{\Mod}[1]{(\textrm{mod}\; #1)}

\theoremstyle{plain}
  \newtheorem{theorem}{Theorem}
  \newtheorem{corollary}[theorem]{Corollary}
  \newtheorem{lemma}[theorem]{Lemma}
  
\theoremstyle{definition}  
  \newtheorem{definition}[theorem]{Definition}

\newtheorem{innercustomthm}{Theorem}
\newenvironment{customthm}[1]
  {\renewcommand\theinnercustomthm{#1}\innercustomthm}
  {\endinnercustomthm}

\title{Static and Dynamical, Fractional Uncertainty Principles}

\author[1]{Kumar, S.\thanks{skumar@bcamath.org}}
\author[1]{Ponce-Vanegas, F.\thanks{fponce@bcamath.org}}
\author[1,2]{Vega, L.\thanks{lvega@bcamath.org}}
\affil[1]{BCAM - Basque Center for Applied Mathematics}
\affil[2]{Universidad del País Vasco UPV/EHU}

\date{}

\begin{document}
\maketitle

\begin{abstract}
We study the process of dispersion of low-regularity 
solutions to the Schrödinger equation using fractional 
weights (observables). 
We give another proof of the uncertainty principle for fractional weights and use it to get a lower bound for the concentration of mass. 
We consider also the evolution when the initial datum is the Dirac comb in $\R$. 
In this case we find fluctuations that 
concentrate at rational times and that resemble a realization of a Lévy process. 
Furthermore, the evolution exhibits multifractality.
\end{abstract}

% !TEX root = Sigma_Space.tex

\section{Introduction}

This work grew out of the interest in 
understanding the process of dispersion of solutions 
to the Schr\"odinger equation with initial data
with low regularity. 
By Schr\"odinger equation we mean the following initial value problem:
\begin{equation*}
\begin{cases}
\partial_t u = \frac{i}{2}\hbar\Delta u \\
u(x,0) = f(x),
\end{cases}
\end{equation*}
where $\hbar := 1/(2\pi)$. 

We measure regularity using the space 
\begin{equation}\label{eq:def:Sigma_delta}
\Sigma_\delta(\R^n) := \{f\in L^2(\R^n)\mid \norm{f}_{\Sigma_\delta}^2 := \norm{\abs{x}^\delta f}^2_2+\norm{D^\delta f}_ 2^2 < \infty\},
\end{equation}
where $D^\delta f:= |\xi|^{\delta} \hat f(\xi)$ and
\begin{equation*}
\hat f(\xi):=\int_{\R^n} e^{-2\pi ix\xi}f(x)\,dx.
\end{equation*}
We will consider $0<\delta\leq1$, and refer to solutions with $u(x,0)\in \Sigma_\delta(\R^n)$, for $0<\delta<1$, as low-regularity solutions. 

Similarly, we measure the dispersion of a solution $u$ with the functional  
\begin{equation}\label{eq:def_h_delta}
h_\delta[f](t) := \int \abs{x}^{2\delta}\abs{u(x,t)}^2\,dx;
\end{equation}
for simplicity, we may write $h_\delta(t)$. Nahas and Ponce studied this functional during their work on persistence properties of decay and regularity in the non-linear setting \cite{zbMATH05651277}. As a consequence of Lemma 2 in \cite{zbMATH05651277} we have
\begin{equation}\label{eq:Nahas-Ponce}
h_\delta[f](t)\leq C_\delta \norm{f}_{\Sigma_\delta(\R^n)}^2(1+t^2)^{\delta},
\end{equation}
where $f$ is the initial datum, so the functional \eqref{eq:def_h_delta} makes sense for every time.
Another proof of this persistence property is given in \cite{agirrePhD}, where the motivation is to give sufficient conditions for uniqueness of linear and non-linear Schr\"odinger equations following the ideas in \cite{zbMATH06124309}.

From another point of view, $h_\delta[f](t)$ is the evolution 
of the average value of a quantum observable 
and the corresponding quantity for a classical particle in free-motion is $h_\delta^{\textrm{c}}[x_0,p_0](t) := \abs{x_0+p_0t}^{2\delta}$, 
where $x_0$ and $p_0$ are the initial position and momentum, respectively. 
It is interesting to compare the quantum and classical behavior; for example, 
after computing $h_1''$ or by using the identity $e^{i\pi t\abs{\xi}^2}(i\hbar\partial)e^{-i\pi t\abs{\xi}^2} = i\hbar\partial + t\xi$ we can see that
\begin{equation*}
h_1[f](t) = \japan{(x_0 + p_0t)^2} := \int f(x)\conj{(x-it\hbar \partial)^2 f}\,dx,
\end{equation*}
where $x_0 = x$ and $p_0 = -i\hbar\partial$ are the initial (in the Heisenberg picture) position and momentum operators, respectively. Does this simple and smooth behavior hold equally when $0<\delta<1$?

The classical Heisenberg's Uncertainty Principle asserts that
\begin{equation}\label{eq:UP_2}
\Big[\int\abs{x}^2\abs{f(x)}^2\,dx\,\int\abs{\xi}^2\abs{\hat f(\xi)}^2\,dx\Big]^\frac{1}{2} \ge \frac{n}{4\pi}\norm{f}_{L^2(\R^n)}^2.
\end{equation}
Using translations in physical space and in phase space (i.e. Galilean transformations) it is always possible to assume that $\int x\abs{f(x)}^2\,dx = \int\xi\abs{\hat f(\xi)}^2\,d\xi = 0$, and 
\eqref{eq:UP_2} is then a measure of the concentration of $\abs{f}$ and $\abs{\hat f}$ around the origin. Finally, using translations in time and dilations  we can also assume that $\japan{x_0p_0 + p_0x_0} = 0$ and that $a^2:=\japan{ x_0^2}= \japan{ p_0^2}$, 
so in that case $h_1[f](t) = a^2(1+t^2)$. Hence, using \eqref{eq:UP_2} we conclude that if $\norm{f}_{L^2(\R^n)}=1$ then
\begin{equation}\label{eq:LB_2}
h_1[f](t) \geq \frac{n}{4\pi}(1+t^2),
\end{equation}
and the identity holds if and only if $f=cf_G(x):= c2^{n/4}e^{-\pi |x|^2}$, where $\abs{c} = 1$. In fact, in that case the corresponding solution is explicitly given by $u_G=2^{n/4}(1+it)^{-n/2}e^{-\pi|x|^2/(1+it)}$, so 
\begin{equation}\label{eq:f_G}h_\delta[f_G](t)=h_\delta[f_G](0)(1+t^2)^\delta.
\end{equation} 

The above argument suggests that a lower bound of 
$h_\delta[f](t)$ might be proved by means of a generalization of the uncertainty principle \eqref{eq:UP_2} with weights $\abs{x}^{2\delta}$ and $\abs{\xi}^{2\delta}$, for $0<\delta<1$. 
As it is well known, the uncertainty principle has been already extended in several directions, see \textit{e.g.} \cite{Cowling1984,zbMATH00788512,zbMATH01038862,zbMATH01921129, zbMATH05237721, Steinerberger2020}, and 
the ``fractional uncertainty principle'' we are interested in was proved by Hirschman in \cite{zbMATH03130602}. One of the results in this paper is another proof of this fact.

\begin{theorem}[Static, Fractional Uncertainty Principle]\label{thm:Static_UP}
There exists a constant $a_\delta>0$, for $0<\delta<1$, such that
\begin{equation}\label{eq:sharp_C_UP}
\inf_{\norm{f}_2 = 1}\norm{\abs{x}^\delta f}_{L^2(\R^n)}\norm{D^\delta f}_{L^2(\R^n)} = a_\delta^2.
\end{equation}
Equality is attained and a minimizer $Q_\delta$ can be chosen strictly positive and satisfying $\norm{\abs{x}^\delta Q_\delta}_2 = \norm{D^\delta Q_\delta}_2$. Any other minimizer $f$ is of the form $f(x) = c\lambda^{n/2}Q_\delta(\lambda x)$ for some $\lambda>0$ and $\abs{c} = 1$. Furthermore, $Q_\delta(x)\simeq \abs{x}^{-n-4\delta}$ for $\abs{x}\gg 1$.
\end{theorem}

The decay result is direct consequence of the work of Kaleta and Kulczycki \cite{zbMATH05806832}. 
Observe that the minimizer of the fractional uncertainty principle does not decay exponentially.

As a consequence of the above theorem we can easily obtain a lower bound for $h_\delta[f](t)$ as stated in our next theorem.

\begin{theorem}[Dynamical, Fractional Uncertainty Principle]\label{thm:Dynamical_UP}
If $f\in\Sigma_\delta(\R^n)$, for $0<\delta<1$, and $\norm{f}_2 = 1$, then
\begin{equation*}
h_\delta[f](t) \ge \Big(\frac{a_\delta^2}{\norm{\abs{x}^\delta f}_2\norm{D^\delta f}_2}\Big)^2\max\Big(\norm{\abs{x}^\delta f}^2_2,\norm{D^\delta f}^2_2\abs{t}^{2\delta}\Big),
\end{equation*} 
where $a_\delta$ is the constant in \eqref{eq:sharp_C_UP}. Furthermore, for any $T \neq 0$
\begin{equation*}
h_\delta[f](0)h_\delta[f](T)\ge a_\delta^4\abs{T}^{2\delta},
\end{equation*} 
with equality if and only if
\begin{equation*}
f(x) = ce^{-\pi i\abs{x}^2/T}\lambda^{n/2}Q_\delta(\lambda x)
\end{equation*} 
for some $\lambda>0$ and $\abs{c} = 1$.
\end{theorem}

One could wonder up to what extent the behavior exhibited by the gaussian in \eqref{eq:f_G} is generic for $h_\delta[f](t)$. One of the main purposes of this paper is to start to explore the answer to this question. 

In Section~\ref{sec:dynamical} we study the regularity of $h_\delta[f](t)$. Theorem~\ref{thm:h_away_zero} is one of our main tools, and in this theorem we compute the Fourier transform of $h_\delta[f](t)$ away from zero. Since the weight is $\abs{x}^{2\delta}$, one may expect that $h_\delta$ has Lipschitz regularity $2\delta$ and that it is best possible in general; in Theorem~\ref{thm:Lipschitz} we confirm this when $n\ge 2$. However, when $n = 1$ and $\delta > 1/2$ the Lipschitz regularity drops to $\frac{1}{4} + \frac{3}{2}\delta$.

In Theorem~\ref{thm:h_decay} we supplement our results about regularity with the analysis of the decay of $\hat{h}_\delta$ . In dimension $n\ge 4$ we prove that $\abs{\hat{h}_\delta(\tau)}\le \tau^{-1-2\delta}$, which is consistent with the Lipschitz regularity. However, when $n\le 3$ the Fourier transform decays more slowly. 

From the analysis of the Fourier transform of $h_\delta$ one can guess that the so called Talbot effect can generate plenty of fluctuations from the generic behavior $(1+t^2)^\delta$, so we turn our attention to it. 

In the torus, the fundamental solution $u_D$ of the Schrödinger equation (when the initial datum is $\delta_0$) has a very peculiar behavior which influence the evolution of any other function. The Talbot effect might be described as the emergence, at rational times, of equally spaced Dirac deltas in the fundamental solution, each Dirac delta being multiplied by a complex weight depending on the Gauss sums. In $\R^n$ we can think of $u_D$ as the solution when the initial datum is the Dirac comb.

The Talbot effect was discovered by Henry Talbot when experimenting with ray lights crossing a grating with equidistant slits \cite{Talbot1836}. The effect was theoretically investigated in several papers by Berry and collaborators, \textit{e.g.} \cite{MR602111,MR1410175,MR1421017}, who discovered various interesting properties.

In a deep connection, the Talbot effect appears during the evolution of a polygonal vortex filament in a fluid through the binormal flow $X_t = \kappa \vc{b}$. Here $X$ is an arc-length parametrization of the filament, $\kappa$ is the curvature and $\vc{b}$ is the binormal vector. In the case of a polygonal filament, the curvature is zero except for the corners where $\kappa$ has Dirac deltas, so $\kappa$ may be seen as a weighted Dirac comb, the weights depending on the angle of the corner. In \cite{MR3363420} Hasimoto discovered a transformation that carries a solution of the binormal flow (curvature and torsion) to a solution of the 1D cubic non-linear Schr\"odinger equation (NLS), so we can use a solution of the NLS with the Dirac comb as initial datum to study the evolution of a polygonal vortex filament. We also notice that a version of the Talbot effect was observed for the NLS numerically \cite{MR3094557}, and then supported theoretically \cite{MR3910065}; see also  Ch. 5.3 of \cite{zbMATH06595353} or \cite{MR3228622, MR3342097} for further information.

In \cite{MR3291141}  the Talbot effect is exploited to show that a polygonal vortex filament evolves into a polygon at rational times as long as uniqueness holds for the binormal flow. They found compelling numerical evidence that $\int_0^t u_D(0,s)\,ds$ approximates the trajectory of a corner in a vortex filament, which has found more theoretical support recently \cite{Banica2020}. Furthermore, the dynamics displays the same conspicuous fractal nature described by Berry and collaborators with respect to the Talbot effect; see also \cite{MR1752508}.

Inspired by ideas in theory of turbulence \cite{MR1428905}, one may wonder whether $\int_0^t u_D(0,s)\,ds$ can be seen as the outcome of a stochastic process. As far as we know, such a description is still missing; however, as we will show in Section~\ref{sec:Dirac_Comb}, when the curves $t\mapsto u_D(x,t)$ are ``averaged'' over $x\in\R$, such a description as a stochastic process might be possible.

Thus, as a second step, we focus our attention to one space dimension and to the particular case when $f$ is the Dirac comb 
\begin{equation*}
F_D(x) := \sum_{m\in\Z}\delta(x-m).
\end{equation*}
Even though $F_D$ is not a proper function but a distribution,
so at first $h_\delta[F_D]$ does not make sense,
we are able to extend, after renormalization, the functional $h_\delta$ to periodic functions
and then to the Dirac comb.
To approach the Dirac comb in $\R$ we use functions of the form
\begin{equation}\label{f12}
f_{\varepsilon_1,\varepsilon_2}(x) := N_{\varepsilon_2}^{-1} \psi(\varepsilon_2 x)F_{\varepsilon_1}/\norm{F_{\varepsilon_1}}_2,
\end{equation}
where $\psi$ is a smooth function with $\psi(0) = 1$, $N_{\varepsilon_2}$ is chosen so that $\norm{f_{\varepsilon_1,\varepsilon_2}}_2 = 1$, and
\begin{equation*}
F_{\varepsilon_1}(x) := \sum_{m\in\Z}\varepsilon_1^{-1}e^{-\pi ((x-m)/\varepsilon_1)^2} = \sum_{m\in\Z} e^{-\pi(\varepsilon_1m)^2}e^{2\pi i x m}.
\end{equation*} 
We will prove that in the limit $\varepsilon_2\to 0$ ($\varepsilon_1$ fixed) the function $h_\delta[f_{\varepsilon_1,\varepsilon_2}]$ splits into a smooth background and a oscillating, periodic function that we call $h_{\textrm{p},\delta}[F_{\varepsilon_1}]$. In Figure~\ref{fig:Asymptotic} we can see how $h_\delta[f_{\varepsilon_1,\varepsilon_2}]$ approaches, after renormalization, $h_{\textrm{p},\delta}[F_{\varepsilon_1}]$. 

\begin{figure}[t]
\centering
\includegraphics[scale=0.7]{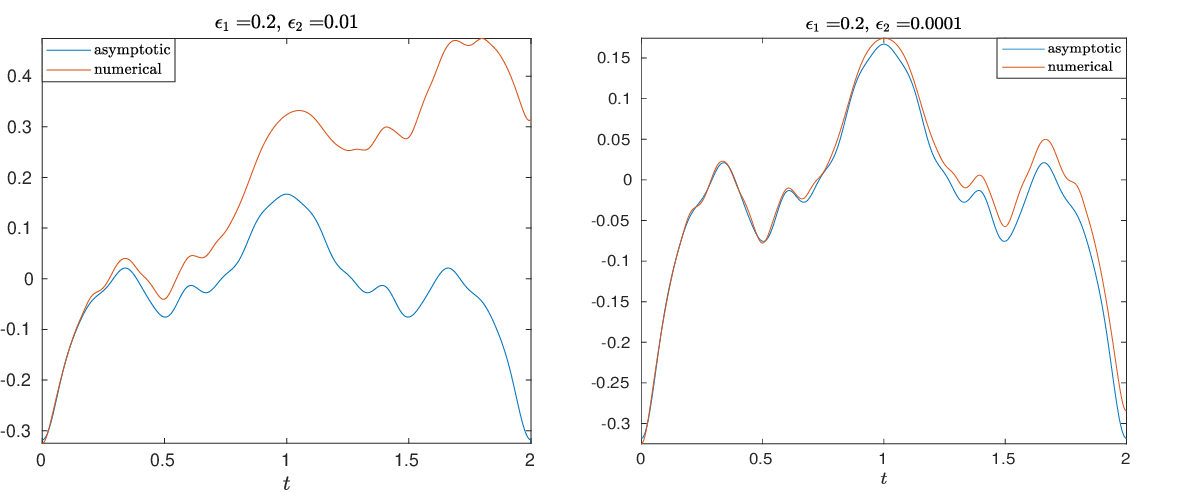}
\caption{The red line is the plot of $h_\delta[f_{\varepsilon_1,\varepsilon_2}]$, for $\delta = 0.25$, using its definition in \eqref{eq:def_h_delta}, and the blue line is $h_{\textrm{p},\delta}[F_{\varepsilon_1}]$, to be defined in \eqref{eq:def:periodic_other}. In this plot we have removed from $h_\delta[f_{\varepsilon_1,\varepsilon_2}]$ a constant term $C_{\varepsilon_2}$ and then multiplied by $\varepsilon_2^{-1}$; this will be clear when we reach \eqref{eq:thm:asymptotic_Smooth_F_1D}. The choice of $\varepsilon_1 = 0.2$ is due to the high computational cost of taking a smaller value of $\varepsilon_1$ and then to diminish $\varepsilon_2$.}\label{fig:Asymptotic}
\end{figure}

The final step is to pass to the limit $\varepsilon_1\to 0$. In this way we obtain a periodic, pure point distribution $h_{\textrm{p},\delta}[F_D]$ with support at rational times, a fact which is very reminiscent of the Talbot effect. More concretely, we prove the following result; see Fig.~\ref{fig:Dirac_deltas}.

\begin{figure}[t]
\centering
\includegraphics[scale=0.8]{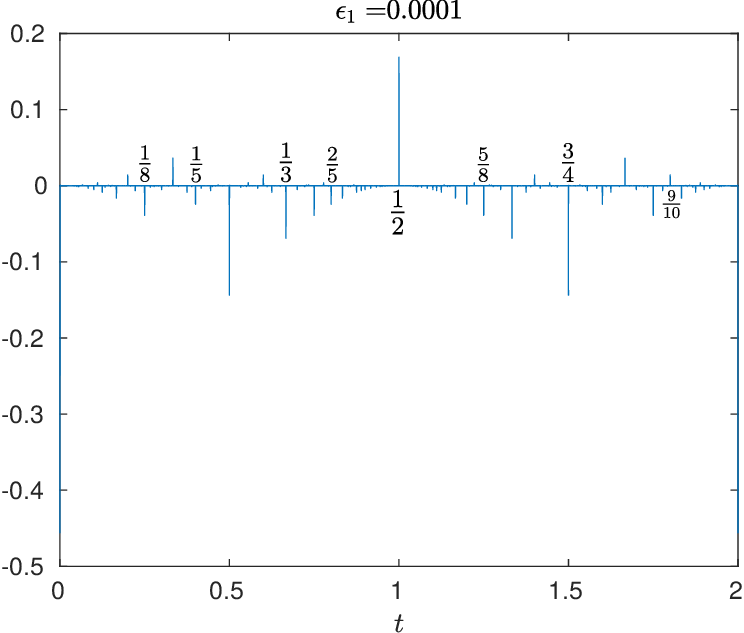}
\caption{Plot of $h_{\textrm{p},\delta}[F_{\varepsilon_1}]$, to be defined in \eqref{eq:def:periodic_other}, when $\delta = 0.25$. In Figure~\ref{fig:Asymptotic} the plot of $h_{\textrm{p},\delta}[F_{\varepsilon_1}]$ lacks the rich structure suggested by \eqref{eq:thm:limit_FD} because $\varepsilon_1$ is still small there; however, as $\varepsilon_1$ approaches zero the emergence of Dirac deltas is clearly visible.}\label{fig:Dirac_deltas}
\end{figure}

\begin{theorem}\label{thm:limit_FD}
\begin{equation}\label{eq:thm:limit_FD}
\begin{split}
h_{\textrm{p},\delta}[F_D](2t) &= -\frac{2b_{1,\delta}}{\norm{\psi}_2^2}\zeta(2(1+\delta))\Big[\sum_{\substack{(p,q)=1 \\ q>0\textrm{ odd}}}\frac{1}{q^{2(1+\delta)}}\delta_\frac{p}{q}(t) - \\
&\hspace*{0.5cm}- \sum_{\substack{(p,q)=1 \\ q\equiv 2\,\Mod 4}}\frac{2(2^{1+2\delta}-1)}{q^{2(1+\delta)}}\delta_\frac{p}{q}(t) + \sum_{\substack{(p,q)=1 \\ q\equiv 0\,\Mod 4}}\frac{2^{2(1+\delta)}}{q^{2(1+\delta)}}\delta_\frac{p}{q}(t)\Big],
\end{split}
\end{equation}
where $\zeta(s)$ is the Riemann zeta function, and
\begin{equation*}
b_{1,\delta} = \frac{1}{(2\pi)^{2\delta}}\frac{\Gamma(2\delta)}{\abs{\Gamma(-\delta)}\Gamma(\delta)}.
\end{equation*}
\end{theorem}

\begin{figure}[t]
\centering
\includegraphics[scale=0.8]{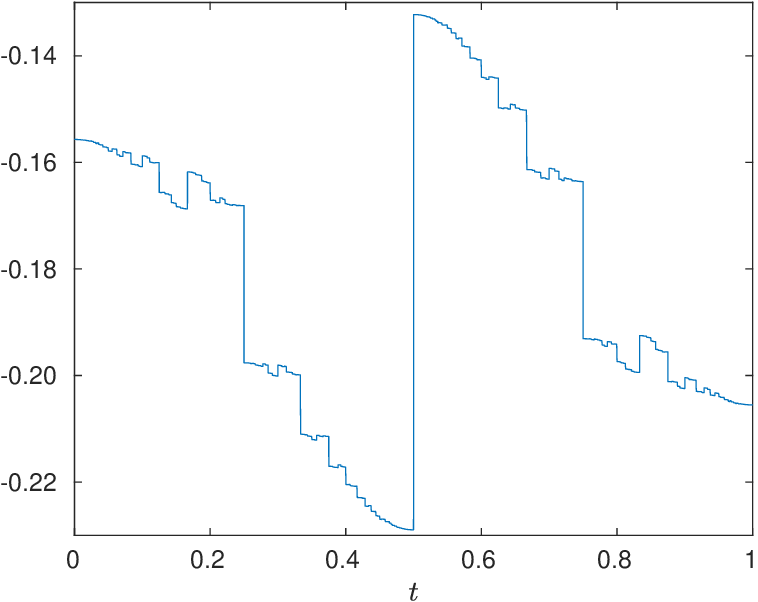}
\caption{Plot of $H_\delta$ in  \eqref{eq:intro:primitive_H}. Even though $H_\delta$ has some symmetry, \textit{e.g.} $H_\delta(1-t) = c_\delta-H_\delta(t-)$, the appearance of ``unpredictable'' large jumps resembles an $\alpha$-Lévy process with small exponent $\alpha$.}\label{fig:Hdelta}
\end{figure}

Our final result is about the properties of $h_{\textrm{p},\delta}[F_D]$. Let us consider its primitive, that is,
\begin{equation}\label{eq:intro:primitive_H}
H_\delta(t) := \int_{[0,t]}h_{\textrm{p},\delta}(2s)\,ds.
\end{equation} 
As we anticipated when discussing the Talbot effect, $H_\delta$ resembles the outcome of a stochastic process, or more precisely, a pure jump $\alpha$-L\'evy process with $\alpha:= 1/(1+\delta)$ (see Fig.~\ref{fig:Hdelta}), which strongly suggests the presence of intermittency. We discuss this in more detail in Section~\ref{sec:Dirac_Comb}. 

Again, inspired by the theory of turbulence (Ch. 8 of \cite{MR1428905}), 
we compute the H\"older exponent of $H_\delta$ at each irrational time and 
show that it depends on its ``irrationality'' $\mu(t)$; 
the precise definition of $\mu(t)$ is given in Definition~\ref{def:Irrationality}. 
We also look
at the so called spectrum of singularities $d_{H_\delta}(\gamma) := \textrm{dim}\,F_\gamma$, where
\begin{equation} \label{eq:def:spectrum_sing}
F_\gamma := \{t\in [0,1)\mid H_\delta \textrm{ has H\"older exponent } \gamma \textrm{ at } t\}.
\end{equation} 
By convention $d_{H_\delta}(\gamma) = -\infty$ if $F_\gamma = \emptyset$.
Our main result in this direction is the following one.

\begin{theorem}\label{thm:intro:Holder_H}
Let $H_\delta$ be the function in \eqref{eq:intro:primitive_H} 
and set $\alpha:= 1/(1+\delta)$.
Then,
\begin{equation} \label{eq:thm:SpecSingH}
d_{H_\delta}(\gamma) = 
\begin{cases}
\alpha\gamma, & \textrm{if }\gamma\in [0,1/\alpha], \\
-\infty, & \textrm{if }\gamma > 1/\alpha.
\end{cases}
\end{equation}
\end{theorem}

Jaffard proved in Theorem~1 of \cite{zbMATH01332836} that 
the spectrum of singularities of an $\alpha$-Lévy process is almost surely equal to \eqref{eq:thm:SpecSingH}. 
This identity tightens our suggested relationship between $H_\delta$ and Lévy processes.
\\[2mm]
\indent{\bf Structure of the paper:}
\begin{itemize}
\item In Section~\ref{sec:Static} we discuss the static, fractional uncertainty principle (Theorem~\ref{thm:Static_UP}) and prove some properties of the space $\Sigma_\delta(\R^n)$. 
\item In Section~\ref{sec:dynamical} we discuss the dynamical, fractional uncertainty principle (Theorem~\ref{thm:Dynamical_UP}) and we compute the Fourier transform of $h_\delta[f]$ (Theorem~\ref{thm:h_away_zero}). In Section~\ref{sec:regularity} we exploit Theorem~\ref{thm:h_away_zero} to obtain regularity properties of $h_\delta[f]$.
\item In Section~\ref{sec:periodic_data} we define $h_\delta[f]$ for periodic initial data. In Section~\ref{sec:Dirac_Comb} we study the ``dispersion'' properties of the Dirac comb, and prove Theorems~\ref{thm:limit_FD} and \ref{thm:intro:Holder_H}.
\end{itemize}

Finally, some questions that arise naturally for future work are:

\begin{enumerate}[nolistsep]
\item  What are the optimal constants in Theorems~\ref{thm:Static_UP} and \ref{thm:Dynamical_UP}? Can $h_\delta[Q_\delta]$ be explicitly computed?
\item What are the results about the Dirac Comb in the non-linear setting?
\item Study different regimes for  $\varepsilon_1$ and $\varepsilon_2$ in \eqref{f12};
\item For other observables (weights) $W(x)$, can we estimate $\japan{e^{-it\hbar\Delta/2}We^{it\hbar\Delta/2}}$ in terms of classical trajectories  $W(x+tp)$?
\end{enumerate}

\subsection*{Notation}

\begin{itemize}
\item \textit{Relations:} If $x\lesssim y$ then $x\le Cy$ for some constant $C>0$, and similarly for $x\gtrsim y$ and $x\simeq y$. If $x\ll 1$ then $x\le c$, where $c>0$ is a sufficiently small constant, and similarly for $x\gg 1$.
\item \textit{Miscellaneous:} $a+ := a + \varepsilon$, for $0<\varepsilon\ll 1$. $\japan{x} := (1+\abs{x}^2)^\frac{1}{2}$. $\sgn$ is the sign function. The standard measure on the sphere is denoted by $dS$.
\item If $A\subset\R^n$, then $\abs{A}$ is its Lebesgue measure and $\ind_A$ is the indicator function.
\item The fractional derivative is $(D^\delta f)^\wedge(\xi) := \abs{\xi}^\delta\hat{f}(\xi)$.
\item Let $I\subset\R$ be an interval with center $c(I)$. The projection to frequencies $\abs{\xi}\in I$ is the operator $(P_If)^\wedge(\xi) := \zeta_I\hat{f}(\xi)$, where $\zeta_I(\xi) := \zeta((\xi-c(I))/\abs{I})$ and $\zeta$ is a fixed cutoff of $[-1,1]$.
\item If $X$ is a function space, then $X_\textrm{loc} := \{f\in\mathcal{S}'\mid \zeta f\in X\textrm{ for every } \zeta\in C^\infty_0\}$.
\item Spaces: for $\Sigma_\delta(\R^n)$ see \eqref{eq:def:Sigma_delta}, and for $\Lambda^\alpha(\R^n)$ see \eqref{eq:def:Lipschitz}. $H^s(\R^n)$ is the space of $f\in L^2(\R^n)$ with $D^sf\in L^2(\R^n)$.
\item $h_f(t)$ is the Hölder exponent of a function $f$ at $t\in\R$; see Def.~\ref{def:Holder}). $d_f$ is the spectrum of singularities; see \eqref{eq:def:spectrum_sing}.
\item $\mu(t)$ is the irrationality measure of $t\in \R$; see Def.~\ref{def:Irrationality}.
\end{itemize}

\subsection*{Funding}

This research is supported by the Basque Government through the BERC 2018-2021 program, by the Spanish State Research Agency through BCAM Severo Ochoa excellence accreditation SEV-2017-0718, and by the ERCEA Advanced Grant 2014 669689-HADE. The second author is also supported by the project PGC2018-094528-B-I00.

\subsection*{Acknowledgments}

We thank Daniel Eceizabarrena for many insightful conversations over the course of this work. We also thank the referees for their suggestions, they helped us to improve the exposition of our work.
 
% !TEX root = Sigma_Space.tex

\section{Static, Fractional Uncertainty Principle}\label{sec:Static}

In this section we study the static, fractional uncertainty principle. We prove some general properties of $\Sigma_\delta(\R^n)$, which will play an important role in our investigation of $h_\delta$.

The (static) uncertainty principle asserts that there exists $a_\delta>0$  such that
\begin{equation*}
\norm{\abs{x}^\delta f}_2\norm{D^\delta f}_ 2\ge a_\delta^2\norm{f}_2^2, \qquad \textrm{for }  0<\delta\le 1.
\end{equation*}
Actually, this is equivalent to the continuous embedding $\Sigma_\delta(\R^n)\hookrightarrow L^2(\R^n)$. In fact, let us define
\begin{equation}
a_\delta^2 := \inf_{\norm{f}_2 = 1}\norm{\abs{x}^\delta f}_2\norm{D^\delta f}_2,
\end{equation}
We can exploit the symmetry $f_\lambda(x) := \lambda^\frac{n}{2}f(\lambda x)$ to force the condition $\norm{\abs{x}^\delta f_\lambda}_2 = \norm{D^\delta f_\lambda}_2$ while preserving $\norm{f}_2 = 1$, so that
\begin{equation*}
2a_\delta^2 = \inf_{\substack{\norm{f}_2 = 1 \\ \norm{\abs{x}^\delta f}_2 = \norm{D^\delta f}_2}}2\norm{\abs{x}^\delta f}_2\norm{D^\delta f}_2 \ge  \inf_{\norm{f}_2 = 1}\norm{f}_{\Sigma_\delta}^2
\end{equation*}
On the other hand, $2\norm{\abs{x}^\delta f}_2\norm{D^\delta f}_2\le \norm{f}_{\Sigma_\delta}^2$ implies the reverse inequality $2a_\delta^2 \le \inf_{\norm{f}_2 = 1}\norm{f}_{\Sigma_\delta}^2$, so
\begin{equation*}
2a_\delta^2 = \inf_{\norm{f}_2 = 1}\norm{f}_{\Sigma_\delta}^2.
\end{equation*}

\begin{lemma}\label{thm:SigmaInL2}
The class $\Sigma_\delta(\R^n)$ is a Hilbert space compactly embedded in $L^2(\R^n)$; in particular,
\begin{equation}\label{eq:thm:SigmaInL2}
\norm{f}_2 \le C(\norm{\abs{x}^\delta f}^2_2+\norm{D^\delta f}_ 2^2)^\frac{1}{2}.
\end{equation}
Furthermore, there exists a function $Q_\delta$ with $\norm{Q_\delta}_2 = 1$ such that
\begin{equation}\label{eq:thm:minimizer}
\inf_{\norm{f}_2=1}\norm{f}_{\Sigma_\delta} = \norm{Q_\delta}_{\Sigma_\delta}
\end{equation}
\end{lemma}
\begin{proof}
We choose a sequence of functions $\{f_n\}_n$ with $\norm{f_n}_2 = 1$ that minimizes $\norm{g}_{\Sigma_\delta}$, that is, $\norm{f_n}_{\Sigma_\delta}\to \inf_{\norm{g}_2=1}\norm{g}_{\Sigma_\delta}$.

By the Fr\'echet-Kolmogorov theorem, the sequence $\{f_n\}$ will be relatively compact in $L^2(\R^n)$ if the following two conditions holds uniformly in $n$:
\begin{align*}
\textrm{(1)}&\quad \int_{\abs{x}>R} \abs{f_n}^2\,dx < \varepsilon, \quad \textrm{for every } \varepsilon > 0 \textrm{ and } R\gg 1 \\
\textrm{(2)}&\quad \norm{f_n(\cdot - h) - f_n}_2 < \varepsilon, \quad \textrm{for every } \varepsilon > 0 \textrm{ and } \abs{h}\ll 1.
\end{align*}
The condition (1) follows from
\begin{equation*}
\int_{\abs{x}>R} \abs{f_n}^2\,dx \le R^{-2\delta}\int\abs{x}^{2\delta}\abs{f_n}^2\,dx \lesssim R^{-2\delta}.
\end{equation*}
The condition (2) follows from
\begin{align*}
\norm{f_n(\cdot - h) - f_n}_2^2 &= \int \abs{\hat{f}_n}^2\abs{e^{-2\pi i\xi\cdot h}-1}^2\,d\xi \\
&\le \int_{\abs{\xi}\le \abs{h}^{-\frac{1}{2}}} +\int_{\abs{\xi}>\abs{h}^{-\frac{1}{2}}} \abs{\hat{f}_n}^2\abs{e^{-2\pi i\xi\cdot h}-1}^2\,d\xi \\
&\lesssim \abs{h} + \abs{h}^\delta.
\end{align*}
Hence, we can choose a sub-sequence $\{f_{n_k}\}_k$ that converges in $L^2(\R^n)$ to some function $Q_\delta\in L^2(\R^n)$ with $\norm{Q_\delta}_2 = 1$. 

If $\inf_{\norm{g}_2=1}\norm{g}_{\Sigma_\delta} = 0$ then $\norm{\abs{x}^\delta f_{n_k}}_2\to 0$ and, passing to a sub-sequence if necessary, we see that $f_{n_k}\to 0$ a.e., which contradicts $\norm{Q_\delta}_2 = 1$. Thus, $\inf_{\norm{g}_2=1}\norm{g}_{\Sigma_\delta}>0$ and $\Sigma_\delta(\R^n)$ is continuously embedded in $L^2(\R^n)$, which is \eqref{eq:thm:SigmaInL2}. Incidentally, the proof shows that the ball $\{\norm{g}_{\Sigma_\delta}\le 1\}$ is relatively compact in $L^2(\R^n)$, so the embedding is compact.

We now prove that $Q_\delta\in \Sigma_\delta(\R^n)$. Since $\Sigma_\delta(\R^n)$ is a Hilbert space, we can pass to a sub-sequence, say $\{f_{n_k}\}_k$, that converges weakly to some $f^* \in \Sigma_\delta(\R^n)$. By \eqref{eq:thm:SigmaInL2} every $h\in L^2(\R^n)$ defines a continuous linear map $g \mapsto \int gh$ in $\Sigma_\delta(\R^n)$, then $\int Q_\delta h = \int f^* h$ and $Q_\delta = f^*\in \Sigma_\delta(\R^n)$.
\end{proof}

The minimizer is the ground state of a differential equation.

\begin{lemma}
If $\norm{Q_\delta}_{\Sigma_\delta} = \inf_{\norm{u}_2=1}\norm{u}_{\Sigma_\delta}$ and $\norm{Q_\delta}_2 = 1$, then
\begin{equation}\label{eq:thm:Eigen_Q_delta}
D^{2\delta}Q_\delta + \abs{x}^{2\delta}Q_\delta = 2a_\delta^2 Q_\delta.
\end{equation}
\end{lemma}
\begin{proof}
We take $v\in \Sigma_\delta(\R^n)$, with $\norm{v}_2 = 1$, orthogonal to $Q_\delta$ in $L^2(\R^n)$. Let us define $w(\theta) := \cos\theta\, Q_\delta + \sin\theta\,v$ so that $f(\theta) := \norm{w(\theta)}_{\Sigma_\delta}^2$ has a minimum at $\theta = 0$. Since the derivative is
\begin{equation*}
f'(\theta) = \sin(2\theta)\big(\norm{v}_{\Sigma_\delta}^2-2a_\delta^2\big) + 2\cos(2\theta)(Q_\delta,v)_{\Sigma_\delta},
\end{equation*}
then $(Q_\delta,v)_{\Sigma_\delta} = 0$; considering $\tilde v = v/\norm{v}_2$, we can remove the condition $\norm{v}_2 = 1$.

For any $v\in \Sigma_\delta(\R^n)$ the function $Pv = v - (Q_\delta,v)_2Q_\delta$ is orthogonal to $Q_\delta$ in $L^2(\R^n)$, so we have $(Q_\delta, Pv)_{\Sigma_\delta} = 0$ or
\begin{equation*}
(Q_\delta,v)_{\Sigma_\delta} = 2a_\delta^2(Q_\delta,v)_2,
\end{equation*}
which is \eqref{eq:thm:Eigen_Q_delta}.
\end{proof}

By the Perron--Frobenius theorem (see Ch. XIII.12 of \cite{zbMATH03622441}) the lowest eigenvalue of the operator $D^{2\delta} + \abs{x}^{2\delta}$ has multiplicity one and the ground state can be chosen strictly positive. To apply this method we need to know that the heat kernel $e^{-tD^{2\delta}}$ is positive; see \cite{zbMATH05779489} or Lemma~A.1 in \cite{zbMATH06199562}. Uniqueness implies that $\hat{Q}_\delta = Q_\delta$ and that $Q_\delta$ is radial.

In Corollary 3 of \cite{zbMATH05806832}, Kaleta and Kulczycki proved that the ground state satisfies $Q_\delta(x) \simeq 1/\abs{x}^{n+4\delta}$ ($0<\delta<1$) for $\abs{x}\gg 1$. 

We summarize the discussion so far in the following theorem, which was stated in the introduction.

\begin{customthm}{\ref{thm:Static_UP}}
{\it
There exists a constant $a_\delta>0$, for $0<\delta<1$, such that
\begin{equation}\label{eq:Heisenberg_delta}
\inf_{\norm{f}_2 = 1}\norm{\abs{x}^\delta f}_{L^2(\R^n)}\norm{D^\delta f}_{L^2(\R^n)} = a_\delta^2.
\end{equation}
Equality is attained and a minimizer $Q_\delta$ can be chosen strictly positive and satisfying $\norm{\abs{x}^\delta Q_\delta}_2 = \norm{D^\delta Q_\delta}_2$. Any other minimizer $f$ is of the form $f(x) = c\lambda^{n/2}Q_\delta(\lambda x)$ for some $\lambda>0$ and $\abs{c} = 1$. Furthermore, $Q_\delta(x)\simeq \abs{x}^{-n-4\delta}$ for $\abs{x}\gg 1$.
}
\end{customthm}

We now prove a couple of properties of $\Sigma_\delta(\R^n)$.

\begin{lemma}\label{thm:smooth_dense}
The space $C^\infty_0(\R^n)$ is dense in $\Sigma_\delta(\R^n)$.
\end{lemma}
\begin{proof}
We choose a symmetric function $\zeta\in C^\infty_0(\R^n)$ such that $\zeta\ge 0$; we might replace $\zeta$ by $\zeta*\zeta$ to assume also that $\hat{\zeta}\ge 0$. By dilation and multiplication by a constant, we assume that $\zeta(0) = 1$ and $\int\zeta = 1$, and we define $\zeta_\lambda(x):=\zeta(x/\lambda)$. 

First we prove that functions with compact support are dense in $\Sigma_\delta(\R^n)$. We fix $\varepsilon>0$ and choose $R\gg_\varepsilon 1$ such that $\norm{\abs{x}^\delta(1-\zeta_R)f}_{2}<\varepsilon$, so we only have to prove that $\norm{\abs{\xi}^\delta(\hat{f}-(\zeta_R f)^\wedge)}_{2}\lesssim \varepsilon$ for $R\gg 1$.

We choose $\lambda\gg_\varepsilon 1$ such that $\norm{\abs{\xi}^\delta\ind_{\abs{\xi}>\lambda} \hat{f}}_{L^2}<\varepsilon$. Since $(\zeta_R f)^\wedge\to\hat{f}$ in $L^2$, then for $R\gg_{\varepsilon,\lambda} 1$ we have that $\norm{\abs{\xi}^\delta\ind_{\abs{\xi}<2\lambda}[\hat{f}-(\zeta_Ru)^\wedge]}_2<\varepsilon$. By Jensen's inequality $\abs{(\zeta_R f)^\wedge}^2\le \hat{\zeta}_R*\abs{\hat{f}}^2$, so
\begin{align*}
\int_{\abs{\xi}>2\lambda} \abs{\xi}^{2\delta}\abs{(\zeta_R f)^\wedge}^2\,d\xi &\le \int (\abs{\xi}^{2\delta}\ind_{\abs{\xi}>2\lambda})*\hat{\zeta}_R\,\abs{\hat{f}}^2\,d\xi \\
&\lesssim \frac{1}{R\lambda}\int_{\abs{\xi}<\lambda} \abs{\hat{f}}^2\,d\xi + \int_{\abs{\xi}>\lambda} \abs{\xi}^{2\delta}\abs{\hat{f}}^2\,d\xi \\
&\lesssim \frac{1}{R\lambda}\int \abs{\hat{f}}^2\,d\xi + \varepsilon^2,
\end{align*} 
where we exploited the rapid decay of $\hat{\zeta}_R$. Hence, $\norm{\abs{\xi}^\delta\ind_{\abs{\xi}>2\lambda}(\zeta_R f)^\wedge}_2\le C\varepsilon$ for $R\gg_{\varepsilon,\lambda} 1$, and
\begin{multline*}
\norm{\abs{\xi}^\delta(\hat{f}-(\zeta_R f)^\wedge)}_2 \le \norm{\abs{\xi}^\delta\ind_{\abs{\xi}<2\lambda}[\hat{f}-(\zeta_Rf)^\wedge]}_2 + \\ 
+ \norm{\abs{\xi}^\delta\ind_{\abs{\xi}>2\lambda}\hat{f}}_2 + \norm{\abs{\xi}^\delta\ind_{\abs{\xi}>2\lambda}(\zeta_R f)^\wedge}_2 \le C\varepsilon,
\end{multline*}
which shows that functions with compact support are dense in $\Sigma_\delta(\R^n)$.

A similar, though simpler argument shows that a function $f\in\Sigma_\delta(\R^n)$ with compact support can be approximated by functions $\zeta_\rho*f \in C^\infty_0(\R^n)$.
\end{proof}

The next lemma contains some embedding properties.

\begin{lemma}\label{thm:embedding_Sigma}
If $f\in\Sigma_\delta(\R^n)$, then $f, \hat{f}\in H^\delta(\R^n)\cap L^p(\R^n)$, where $p$ satisfies:
\begin{equation}\label{eq:thm:embedding_Sigma}
\begin{split}
\frac{1}{2}-\frac{\delta}{n}\le \frac{1}{p}<\frac{1}{2}+\frac{\delta}{n} \qquad &\textrm{if } n\ge 2, \textrm{ or } n=1 \textrm{ and } \delta< \frac{1}{2}, \\
0< \frac{1}{p}<1 \qquad &\textrm{if } n=1 \textrm{ and } \delta= \frac{1}{2}, \\
0\le \frac{1}{p}\le 1 \qquad &\textrm{if } n=1 \textrm{ and } \delta > \frac{1}{2}.
\end{split}
\end{equation}
\end{lemma}
\begin{proof}
The inequalities at the left follow from the Sobolev Embedding Theorem, and those at the right follow from Hölder inequality.
\end{proof}

We cannot improve the strict inequalities in \eqref{eq:thm:embedding_Sigma}, and we can use the examples $f(x) := \zeta(\abs{x}) \abs{x}^{-\frac{n}{2}-\delta}(\log\abs{x})^{-\frac{1}{2}-\varepsilon}$, for $0<\varepsilon<\delta/n$, where $\zeta\in C^\infty(\R)$ vanishes around zero. When $n=1$ and $\delta = 1/2$, it is known that $f$ may not be bounded.

% !TEX root = Sigma_Space.tex

\section{Dynamical, Fractional Uncertainty Principle}\label{sec:dynamical}

In this section we turn our attention to $h_\delta[f]$ in \eqref{eq:def_h_delta}. We begin with a lower bound for $h_\delta[f]$, and then we compute the Fourier transform of $h_\delta[f]$ away from the origin. In Section~\ref{sec:regularity} we determine the Hölder regularity of $h_\delta$ and the rate of decay of $\hat h_\delta$. 

\begin{customthm}{\ref{thm:Dynamical_UP}}
{\it
If $f\in\Sigma_\delta(\R^n)$, for $0<\delta<1$, and $\norm{f}_2 = 1$, then
\begin{equation}\label{eq:thm:Lower_Bound}
h_\delta[f](t) \ge \Big(\frac{a_\delta^2}{\norm{\abs{x}^\delta f}_2\norm{D^\delta f}_2}\Big)^2\max\Big(\norm{\abs{x}^\delta f}^2_2,\norm{D^\delta f}^2_2\abs{t}^{2\delta}\Big),
\end{equation} 
where $a_\delta$ is the constant in \eqref{eq:Heisenberg_delta}. Furthermore, for any $T \neq 0$
\begin{equation}\label{eq:thm:two_times}
h_\delta[f](0)h_\delta[f](T)\ge a_\delta^4\abs{T}^{2\delta},
\end{equation} 
with equality if and only if
\begin{equation}\label{eq:thm:Dynamical_minimizer}
f(x) = ce^{-\pi i\abs{x}^2/T}\lambda^{n/2}Q_\delta(\lambda x)
\end{equation} 
for some $\lambda>0$ and $\abs{c} = 1$.
}
\end{customthm}
\begin{proof}
The solution $u$ can be represented as
\begin{equation*}
u(x,t) = \frac{1}{(it)^\frac{n}{2}}e^{\pi i\abs{x}^2/t}\int f(y)e^{\pi i\abs{y}^2/t - 2\pi i x\cdot y/t}\,dy, \qquad\textrm{where } \real{\sqrt{it}}>0. 
\end{equation*}
If we define $g_t(y) := f(y)e^{\pi i\abs{y}^2/t}$, then the solution can be written as 
\begin{equation*}
u(x,t) = \frac{1}{(it)^\frac{n}{2}}e^{\pi i\abs{x}^2/t}\hat{g}_t(x/t).
\end{equation*}
By the uncertainty principle \eqref{eq:Heisenberg_delta} we have
\begin{equation*}
a_\delta^2 \le \norm{\abs{x}^\delta g_t}_2\norm{D^\delta g_t}_2 = \abs{t}^{-\delta}h_\delta(0)^\frac{1}{2}h_\delta(t)^\frac{1}{2},
\end{equation*}
with equality if and only if $g_t(x) = c\lambda^{n/2}Q_\delta(\lambda x)$ for some $\lambda>0$ and $\abs{c} = 1$, so \eqref{eq:thm:two_times} and \eqref{eq:thm:Dynamical_minimizer} hold. This inequality implies the lower bound
\begin{equation}\label{eq:h_Lower_A}
h_\delta(t) \ge \frac{a_\delta^4}{\norm{\abs{x}^\delta f}^2_2}\abs{t}^{2\delta}.
\end{equation}
On the other hand, again by \eqref{eq:Heisenberg_delta}, we have
\begin{equation*}
a_\delta^4 \le h_\delta(t)\int\abs{\xi}^{2\delta}\abs{\hat{u}(\xi,t)}^2\,d\xi = h_\delta(t)\int\abs{\xi}^{2\delta}\abs{\hat{f}(\xi)}^2\,d\xi,
\end{equation*}
which implies the lower bound
\begin{equation} \label{eq:h_Lower_B}
h_\delta(t) \ge \frac{a_\delta^4}{\norm{D^\delta f}^2_2}.
\end{equation}
From \eqref{eq:h_Lower_A} and \eqref{eq:h_Lower_B} we conclude that
\begin{equation*}
h_\delta(t) \ge \max\Big(\frac{a_\delta^4}{\norm{D^\delta f}^2_2},\frac{a_\delta^4}{\norm{\abs{x}^\delta f}^2_2}\abs{t}^{2\delta}\Big),
\end{equation*}
which is \eqref{eq:thm:Lower_Bound} after reordering.
\end{proof}

%\subsection{The Fourier Transform of $h_\delta$}\label{sec:Fourier_h}

The computation of the Fourier transform of $h_\delta[f]$ is motivated by the oscillations that appear in numerical simulations when $f$ approaches the Dirac comb.

\begin{theorem}\label{thm:h_away_zero}
If $f\in\Sigma_\delta(\R^n)$, then the Fourier transform of $h_\delta[f]$ in $\R\setminus \{0\}$ can be represented as
\begin{equation}\label{eq:thm:Fourier_Transform_h}
\hat{h}_\delta(\tau) = -2b_{n,\delta}\int_{\R^{2n}}\hat{f}(\xi)\conj{\hat{f}}(\eta)\delta_0\Big(\tau-\frac{\abs{\eta}^2-\abs{\xi}^2}{2}\Big)\frac{d\xi d\eta}{\abs{\xi-\eta}^{n+2\delta}},
\end{equation}
where
\begin{equation*}
b_{n,\delta} = \frac{1}{2\pi^{n/2+2\delta}}\frac{\Gamma(\frac{n+2\delta}{2})}{\abs{\Gamma(-\delta)}}. 
\end{equation*} 
If $\varphi\in\mathcal{S}(\R)$ is supported outside the interval $(-a,a)$, then
\begin{equation} \label{eq:thm:Order_Fourier_h}
\abs{\inner{\hat{h}_\delta[f]}{\varphi}} \le C_a\norm{f}_{\Sigma_\delta}^2\norm{\varphi}_\infty.
\end{equation}
Furthermore,
\begin{equation} \label{eq:thm:Fourier_h_L1}
\norm{\hat h_\delta[f]}_{L^1(\R\setminus[-a,a])}\le C_a\norm{f}_{\Sigma_\delta}^2.
\end{equation}
\end{theorem}
\begin{proof}
The computation is based on the identity (see \cite{zbMATH07000807})
\begin{equation*}
h_\delta(t) = b_{n,\delta}\int_{\R^n\times\R^n}\frac{\abs{e^{-\pi it\abs{\xi}^2}\hat{f}(\xi) - e^{-\pi it\abs{\eta}^2}\hat{f}(\eta)}^2}{\abs{\xi-\eta}^{n+2\delta}}\,d\xi d\eta, \quad\textrm{for } 0<\delta<1.
\end{equation*}

Let $\varphi\in\mathcal{S}(\R)$ be a test function that vanishes in the interval $(-a,a)$. We apply Fubini to write the Fourier transform of $h_\delta$ as
\begin{align}
\inner{h_\delta}{\hat{\varphi}} &= b_{n,\delta}\int_{\R^{2n}}\int \big[\abs{\hat{f}(\xi)}^2 + \abs{\hat{f}(\eta)}^2 - e^{-\pi it(\abs{\xi}^2-\abs{\eta}^2)}\hat{f}(\xi)\conj{\hat{f}}(\eta) - \notag \\
&\hspace*{4cm} - e^{-\pi it(\abs{\eta}^2-\abs{\xi}^2)}\conj{\hat{f}}(\xi)\hat{f}(\eta)\big]\hat{\varphi}(t)\,dt\frac{d\xi d\eta}{\abs{\xi-\eta}^{n+2\delta}} \label{eq:h_fourier_pairing_zero} \\
&= -b_{n,\delta}\int_{\R^{2n}}\Big[\hat{f}(\xi)\conj{\hat{f}}(\eta)\varphi\Big(\frac{\abs{\eta}^2-\abs{\xi}^2}{2}\Big) + \notag \\
&\hspace*{4cm} + \conj{\hat{f}}(\xi)\hat{f}(\eta)\varphi\Big(\frac{\abs{\xi}^2-\abs{\eta}^2}{2}\Big)\Big]\frac{d\xi d\eta}{\abs{\xi-\eta}^{n+2\delta}}, \label{eq:h_fourier_pairing}
\end{align}
We have to show that this integral represents a bounded functional in $\mathcal{S}(\R)$.

Assume that $a\le 1$. We bound the integral \eqref{eq:h_fourier_pairing} as
\begin{align}
\abs{\inner{h_\delta}{\hat{\varphi}}} &\le C\int \Big|\hat{f}(\xi)\conj{\hat{f}}(\eta)\varphi\Big(\frac{\abs{\eta}^2-\abs{\xi}^2}{2}\Big)\Big|\frac{d\xi d\eta}{\abs{\xi-\eta}^{n+2\delta}} \notag \\
&\lesssim \int\abs{\hat{f}(\xi)}^2\Big[\int\varphi\Big(\frac{\abs{\eta}^2-\abs{\xi}^2}{2}\Big)\frac{d\eta}{\abs{\xi-\eta}^{n+2\delta}}\Big]d\xi + \notag \\
&\hspace*{3.5cm}+ \int\abs{\hat{f}(\eta)}^2\Big[\int\varphi\Big(\frac{\abs{\eta}^2-\abs{\xi}^2}{2}\Big)\frac{d\xi}{\abs{\xi-\eta}^{n+2\delta}}\Big]d\eta  \notag \\
&:= C\int\abs{\hat{f}(\xi)}^2J(\xi)d\xi + C\int\abs{\hat{f}(\eta)}^2J'(\eta)d\eta, \label{eq:Int_f_J}
\end{align}
where $J$ and $J'$ are the integrals in square brackets. 

We only bound the first integral in \eqref{eq:Int_f_J}, the other being analogous; recall that $\varphi((\abs{\eta}^2-\abs{\xi}^2)/2) = 0$ if $\abs{\abs{\eta}^2-\abs{\xi}^2}< 2a$. When $\abs{\xi}> \sqrt{a}$ we control $J$ as
\begin{equation*}
J(\xi) \le \norm{\varphi}_\infty\int_{\R^n\setminus B(\xi,a/(2\abs{\xi}))}\frac{d\eta}{\abs{\xi-\eta}^{n+2\delta}}\le Ca^{-2\delta}\norm{\varphi}_\infty\abs{\xi}^{2\delta},
\end{equation*}
and when $\abs{\xi}<\sqrt{a}$ we integrate instead over $\R^n\setminus B(\xi,\sqrt{a}/2)$. The final result is
\begin{equation}\label{eq:thm_fourier_h_J}
J(\xi) \le C\norm{\varphi}_\infty\begin{cases}
a^{-2\delta}\abs{\xi}^{2\delta} & \abs{\xi}> \sqrt{a}, \\
a^{-\delta} & \abs{\xi}< \sqrt{a}.
\end{cases}
\end{equation} 

\begin{center}
\includegraphics[scale=0.7]{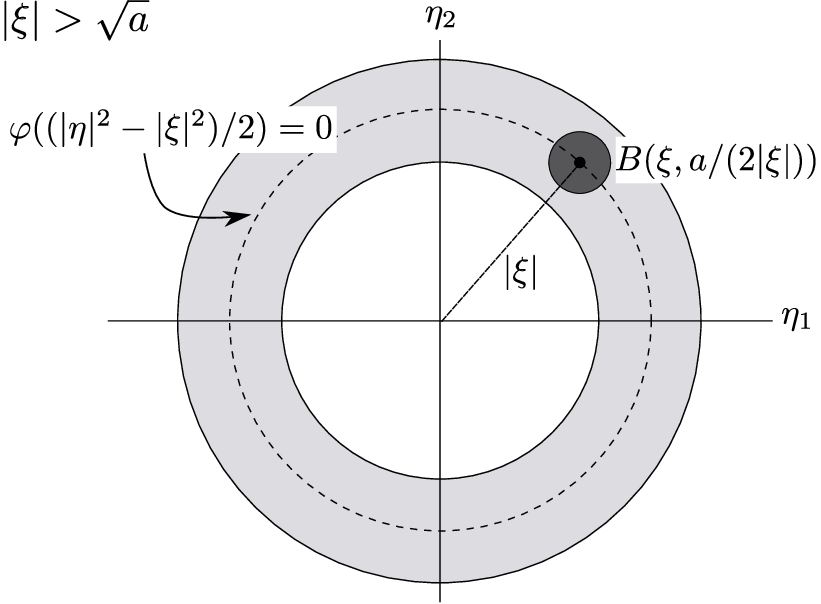}
\end{center}

We replace \eqref{eq:thm_fourier_h_J} in \eqref{eq:Int_f_J} and use the inclusion $\Sigma_\delta\xhookrightarrow{}L^2$ to conclude that
\begin{equation*}
\abs{\inner{h_\delta}{\hat{\varphi}}} \le Ca^{-2\delta}\norm{\varphi}_\infty\norm{f}_{\Sigma_\delta}^2,
\end{equation*}
which is \eqref{eq:thm:Order_Fourier_h}.

Since $\mathcal{S}(\R)$ is dense in the space of continuous functions that vanish at infinity, then from \eqref{eq:thm:Order_Fourier_h} and the Riesz-Markov Theorem we can see $\hat h_\delta$ as a (signed) regular measure in $\R\setminus[-a,a]$ with total variation $\le C_a\norm{f}_{\Sigma_\delta}^2$.

The measure $\hat h_\delta$ is actually an $L^1$-function away from the origin. If $U\subset\R\setminus[-a,a]$ is an open set, then we can  approximate monotonically $\ind_U$ with Schwartz functions $\varphi$ such that $0\le\varphi\le 1$ and $\supp\varphi\subset U$, so by dominated convergence we can write
\begin{equation*}
\inner{\hat h_\delta}{\ind_U} = -2b_{n,\delta}\int_{\R^{2n}}\hat{f}(\xi)\conj{\hat{f}}(\eta)\ind_U\Big(\frac{\abs{\eta}^2-\abs{\xi}^2}{2}\Big)\frac{d\xi d\eta}{\abs{\xi-\eta}^{n+2\delta}}.
\end{equation*}
Since $\hat h_\delta$ is a regular measure, we can actually extend this identity from $\ind_U$ to all bounded, Borel measurable functions. If $A\subset\R\setminus[-a,a]$ is a bounded, Borel set with $\abs{A} = 0$, then we can apply this identity to $\psi\ind_A$, for $\abs{\psi}\le 1$, to conclude that $\hat h_\delta$ is absolutely continuous away from the origin.
\end{proof}

\begin{corollary}\label{corol:h_continuous}
The function $h_\delta$ is continuous.
\end{corollary}
\begin{proof}
We split $h_\delta$ into $P_{<1}h_\delta$ (an analytic function) and $P_{>1}h_\delta$. By \eqref{eq:thm:Fourier_h_L1} $(P_{>1}h_\delta)^\wedge\in L^1(\R)$ and the claim follows.
\end{proof}

\subsection{Regularity of $h_\delta$}\label{sec:regularity}

Corollary \ref{corol:h_continuous} says that $h_\delta$ is continuous; however, we can improve our estimates and refine the information about regularity.

The space of Lipschitz functions $\Lambda^\alpha(\R^n)$, for $\alpha>0$, is
\begin{equation}\label{eq:def:Lipschitz}
\begin{split}
\Lambda^\alpha(\R^n) &:= \{f\in L^\infty(\R^n) \mid \\
&\hspace*{1.5cm} \norm{P_{[2^k,2^{k+1}]}f}_\infty\le C2^{-\alpha k},\textrm{ for } k\ge 0, \textrm{ and } \norm{P_{[0,1]}f}_\infty\le C\}. 
\end{split}
\end{equation}
If $f\in\Lambda^\alpha(\R)$, for $0<\alpha<1$, then $\abs{f(x)-f(y)}\le C\abs{x-y}^\alpha$; see Ch. V.4 of \cite{MR0290095}. 

\begin{theorem}\label{thm:Lipschitz}
If $f\in \Sigma_\delta(\R^n)$, for $0<\delta<1$, then
\begin{equation}\label{eq:thm:regularity}
\norm{\psi h_\delta[f]}_{\Lambda^\alpha}\lesssim C_\psi\norm{f}^2_{\Sigma_\delta}
\end{equation}
where $\psi\in C^\infty_0(\R)$ and
\begin{equation*}
\alpha = \begin{cases}
2\delta & \textrm{for } n\ge 2, \textrm{ or for } n=1 \textrm{ and } \delta<\frac{1}{2}, \\
1- & \textrm{for } n=1 \textrm{ and } \delta=\frac{1}{2}, \\ 
\frac{1}{4}+\frac{3}{2}\delta & \textrm{for } n=1 \textrm{ and } \delta>\frac{1}{2}.
\end{cases}
\end{equation*}
The result is the best possible --- up to the end point in the case $n=1$ and $\delta=\frac{1}{2}$. In particular, $h_\delta\in C^1_\textrm{loc}(\R)$ when $\delta>\frac{1}{2}$.
\end{theorem}
\begin{proof}
Since $P_{\le 1}h_\delta$ and its derivatives are bounded in compact sets by the Nahas-Ponce inequality \eqref{eq:Nahas-Ponce}, then it suffices to prove that $P_{\ge 1}h_\delta\in \Lambda^\alpha(\R)$. Since $h_\delta$ is real, then $\hat h_\delta(\tau) = \conj{\hat h}_\delta(-\tau)$ and we only need to work with positive frequencies. Hence, by the Hausdorff-Young inequality, it suffices to prove
\begin{equation*}
\norm{\hat{h}_\delta}_{L^1(\tau\simeq 2^k)} \le C\norm{f}^2_{\Sigma_\delta} \begin{cases}
2^{-2\delta k} & \textrm{for } n\ge 2, \textrm{ or for } n=1 \textrm{ and } \delta\le\frac{1}{2} \\
2^{-(\frac{1}{4}+\frac{3}{2}\delta)k} & \textrm{for } n=1 \textrm{ and } \delta>\frac{1}{2}.
\end{cases}
\end{equation*}

We define $I_\lambda := [\lambda,2\lambda]$, for $\lambda\ge 1$, and re-scale \eqref{eq:h_fourier_pairing} to get, for $\abs{g}\le 1$,
\begin{equation}\label{eq:h_regularity_Holder_C}
\abs{\inner{\hat{h}_\delta}{g\ind_{I_\lambda}}} \le C\lambda^{\frac{n}{2}-\delta}\int_{\R^{2n}}\abs{\hat{f}(\sqrt{\lambda}\xi)\hat{f}(\sqrt{\lambda}\eta)}\ind_I\Big(\frac{\abs{\eta}^2-\abs{\xi}^2}{2}\Big)\frac{d\xi d\eta}{\abs{\xi-\eta}^{n+2\delta}},
\end{equation}
where $I = [1,2]$.

To bound the integral over the region $\{\abs{\xi}>1\}$, we begin with
\begin{align*}
\abs{\inner{\hat{h}_\delta}{g\ind_{I_\lambda}}}_{\{\abs{\xi}>1\}} &\le C\lambda^{\frac{n}{2}-\delta}\int_{\abs{\xi}>1}\abs{\hat f(\sqrt{\lambda}\xi)}^2\Big[\int \ind_{I\cup -I}\Big(\frac{\abs{\eta}^2-\abs{\xi}^2}{2}\Big)\frac{d\eta}{\abs{\xi-\eta}^{n+2\delta}}\Big]\,d\xi \\
&:= C\lambda^{\frac{n}{2}-\delta}\int_{\abs{\xi}>1}\abs{\hat f(\sqrt{\lambda}\xi)}^2J(\xi)\,d\xi;
\end{align*}
compare with \eqref{eq:Int_f_J}. We use \eqref{eq:thm_fourier_h_J}, for $a=1$, to find out
\begin{align}
\abs{\inner{\hat{h}_\delta}{g\ind_{I_\lambda}}}_{\{\abs{\xi}>1\}} &\le C\lambda^{\frac{n}{2}-\delta}\int\abs{\hat{f}(\sqrt{\lambda}\xi)}^2\abs{\xi}^{2\delta}\,d\xi \notag \\
&\le C\lambda^{-2\delta}\norm{f}_{\Sigma_\delta}^2.\label{eq:h_regularity_Holder_B}
\end{align}

To bound the integral over the region $\{\abs{\xi}<1\}$, we begin with \eqref{eq:h_regularity_Holder_C} and notice that the factor $\ind_I((\abs{\eta}^2-\abs{\xi}^2)/2)$ forces $\abs{\eta}\simeq 1$. Hence,
\begin{align}
\abs{\inner{\hat{h}_\delta}{g\ind_{I_\lambda}}}_{\{\abs{\xi}<1\}} &\le C\lambda^{\frac{n}{2}-\delta}\int_{\abs{\xi}< 1, \abs{\eta}\simeq 1}\abs{\hat{f}(\sqrt{\lambda}\xi)\hat{f}(\sqrt{\lambda}\eta)}\,d\xi d\eta \notag \\
&\le C\lambda^{-\frac{n}{2}-\delta}\int_{\abs{\xi}< \sqrt{\lambda}}\abs{\hat{f}(\xi)}\,d\xi\,\int_{\abs{\eta}\simeq \sqrt{\lambda}}\abs{\hat{f}(\eta)}\,d\eta \notag \\
&\le C\lambda^{-\frac{n}{4}-\frac{3}{2}\delta} \Big(\int_{\abs{\xi}< \sqrt{\lambda}}\abs{\hat{f}(\xi)}\,d\xi\Big)\norm{\abs{\eta}^\delta\hat{f}}_2 \label{eq:h_regularity_Holder_A}.
\end{align}
We control the term in parentheses as
\begin{multline*}
\int_{\abs{\xi}< \sqrt{\lambda}}\abs{\hat{f}(\xi)}\,d\xi \lesssim \Big(\int_{\abs{\xi}<1}\abs{\hat f(\xi)}^2\,d\xi\Big)^\frac{1}{2} + \\
+ \Big(\int_{1<\abs{\xi}<\sqrt{\lambda}}\abs{\xi}^{-2\delta}\,d\xi\Big)^\frac{1}{2}\Big(\int_{1<\abs{\xi}<\sqrt{\lambda}}\abs{\xi}^{2\delta}\abs{\hat f(\xi)}^2\,d\xi\Big)^\frac{1}{2};
\end{multline*}
after replacing in \eqref{eq:h_regularity_Holder_A} we arrive to
\begin{equation*}
\abs{\inner{\hat{h}_\delta}{g\ind_{I_\lambda}}}_{\{\abs{\xi}<1\}} \lesssim \norm{f}_{\Sigma_\delta}^2
\begin{cases}
\lambda^{-2\delta} & \textrm{for } n\ge 2, \textrm{ or } n = 1 \textrm{ and } \delta<\frac{1}{2}, \\
\lambda^{-1}\sqrt{\log \lambda} & \textrm{for } n = 1 \textrm{ and } \delta=\frac{1}{2}, \\
\lambda^{-\frac{1}{4}-\frac{3}{2}\delta} & \textrm{for } n = 1 \textrm{ and } \delta>\frac{1}{2},
\end{cases}
\end{equation*}
which together with \eqref{eq:h_regularity_Holder_B} implies \eqref{eq:thm:regularity}---notice that $\frac{1}{4}+\frac{3}{2}\delta<2\delta$.
\\[0.2cm]
\indent\textit{Sharpness of the regularity}
\\[0.2cm]
We consider functions $\hat{f}_\alpha(\xi) := \japan{\xi}^{-\alpha}$, for $\alpha = \frac{n}{2}+\delta +$. The Fourier transform of $h_\delta[f]$ is symmetric and, for $\tau>0$, it equals
\begin{align*}
\hat{h}_\delta(\tau) &= -2b_{n,\delta}\tau^{\frac{n}{2}-1-\delta}\int \hat{f}_\alpha(\sqrt{\tau}\xi)\hat{f}_\alpha(\sqrt{\tau}\eta)\delta\Big(1-\frac{\abs{\eta}^2-\abs{\xi}^2}{2}\Big)\frac{d\xi d\eta}{\abs{\xi-\eta}^{n+2\delta}}
\end{align*}

Since $\hat{h}_\delta\le 0$, it is enough to prove that $\abs{\hat{h}_\delta(\tau)}\ge c\tau^{-1-\beta}$ for $\abs{\tau}\gg 1$, where
\begin{equation}\label{eq:Lipschitz_beta}
\beta = \begin{cases}
2\delta+ & \textrm{for } n\ge 2, \textrm{ or for } n=1 \textrm{ and } \delta<\frac{1}{2} \\
\frac{1}{4}+\frac{3}{2}\delta+ & \textrm{for } n=1 \textrm{ and } \delta\ge\frac{1}{2}.
\end{cases}
\end{equation}
In fact, if $\{\zeta_I\}$ is a cut-off function of $I := \{2^k\le\abs{\tau}\le 2^{k+1}\}$, then 
\begin{equation*}
c2^{-\beta k} \le \norm{\zeta_I\hat{h}_\delta}_1 = \abs{P_I h_\delta(0)}\le \norm{P_I h_\delta}_\infty,
\end{equation*}
and $\hat h_\delta\notin \Lambda^{\beta+}(\R)_{\textrm{loc}}$.

We use spherical coordinates and bound $\hat{h}_\delta$ from below as
\begin{equation}\label{eq:Lipschitz_exampla_A}
\begin{split}
\abs{\hat{h}_\delta(\tau)} &\ge c\tau^{\frac{n}{2}-1-\delta}\int_{\tau^{-\frac{1}{2}}<r_1<1} \hat{f}_\alpha(\sqrt{\tau}r_1)\hat{f}_\alpha(\sqrt{\tau}r_2)\delta\Big(1-\frac{r_2^2-r_1^2}{2}\Big) \\
&\hspace*{3cm} \Big[\int_{S^{n-1}\times S^{n-1}}\frac{d\theta_1d\theta_2}{\abs{r_1\theta_1-r_2\theta_2}^{n+2\delta}}\Big] r_1^{n-1}r_2^{n-1}dr_1dr_2.
\end{split}
\end{equation}
We denote by $J(r_1,r_2)$ the term inside parentheses; by rotational symmetry
\begin{equation*}
J(r_1,r_2) = c\int_{S^{n-1}}\frac{d\theta}{\abs{r_1e_n - r_2\theta}^{n+2\delta}}.
\end{equation*} 
The term $\delta(1-(r_2^2-r_1^2)/2)$ forces $r_2\simeq 1$, so $J(r_1,r_2)\gtrsim 1$, and from \eqref{eq:Lipschitz_exampla_A} we deduce
\begin{align*}
\abs{\hat{h}_\delta(\tau)} &\ge c\tau^{\frac{n}{2}-1-\delta}\int_{\tau^{-\frac{1}{2}}}^1\frac{1}{\tau^\alpha r_1^\alpha}\int_0^\infty\delta\Big(1-\frac{r_2^2-r_1^2}{2}\Big)\,dr_2\, r_1^{n-1}dr_1 \\
&\ge c\tau^{\frac{n}{2}-1-\delta-\alpha}\int_{\tau^{-\frac{1}{2}}}^1 r_1^{n-1-\alpha}\,dr_1
\end{align*}
Since $\alpha = \frac{n}{2}+\delta+$, we conclude, for $\abs{\tau}\gg 1$, that 
\begin{equation*}
\abs{\hat{h}_\delta(\tau)} \ge c
\begin{cases}
\tau^{-1-2\delta-} & \textrm{for } n\ge 2, \textrm{ or for } n=1 \textrm{ and } \delta<\frac{1}{2} \\
\tau^{-1-\frac{1}{4}-\frac{3}{2}\delta-} & \textrm{for } n=1 \textrm{ and } \delta\ge \frac{1}{2},
\end{cases}
\end{equation*}
which implies \eqref{eq:Lipschitz_beta}.

As a final remark, if $f$ is one of the examples we used, then $h_\delta[f]$, which is an even function, has a singularity at zero of the form $\abs{t}^\rho$. By translation in time, we can place the singularity at any other time.
\end{proof}

We may compare the regularity of $h_\delta$ with its classical counterpart $h_\delta^\textrm{c}[x,\xi](t) := \abs{x+t\xi}^{2\delta}$, which belongs to $\Lambda^{2\delta}_\textrm{loc}(\R)$. If $n\ge 2$ then $h_\delta^\textrm{c}$ is smooth for a general choice of $x$ and $\xi$, but if $n=1$ then $h_\delta^\textrm{c}$ has a singularity for a general choice, which agrees with the loss of regularity in Theorem~\ref{thm:Lipschitz} when $n = 1$.

In the following theorem we investigate the rate of decay of $\hat{h}_\delta$. However, first we have to prove an auxiliary result.

\begin{lemma}\label{thm:aux_decay_h}
Let $n\ge 1$ and let $r_1$ and $r_2$ be different, positive numbers. If  $\alpha > n-1$ and $A,B\in L^2(S^{n-1})$, then
\begin{equation}
\int_{S^{n-1}\times S^{n-1}} \frac{A(\theta_1)B(\theta_2)\,d\theta_1d\theta_2}{\abs{r_1\theta_1-r_2\theta_2}^\alpha} \le C_{r_1,r_2}\norm{A}_{L^2(S^{n-1})}\norm{B}_{L^2(S^{n-1})}
\end{equation}
where
\begin{equation*}
C_{r_1,r_2} \lesssim
\begin{cases}
1/r_1^\alpha & \textrm{for }\; r_1>2r_2, \\
1/r_2^\alpha & \textrm{for }\; r_2>2r_1, \\
(r_1r_2)^{-\frac{n-1}{2}}\abs{r_1-r_2}^{n-1-\alpha} & \textrm{for }\; \frac{1}{2}r_2<r_1<2r_2.
\end{cases}
\end{equation*}
\end{lemma}
\begin{proof}
We assume that $r_2<r_1$ and that $\norm{A}_2 = \norm{B}_2 = 1$, so
\begin{align*}
\int \frac{A(\theta_1)B(\theta_2)\,d\theta_1d\theta_2}{\abs{r_1\theta_1-r_2\theta_2}^\alpha}&\le \frac{1}{2}\sup_{\theta_1}\int\frac{d\theta_2}{\abs{r_1\theta_1-r_2\theta_2}^\alpha}+\frac{1}{2}\sup_{\theta_2}\int\frac{d\theta_1}{\abs{r_1\theta_1-r_2\theta_2}^\alpha} \\
&= \frac{1}{r_2^\alpha}\int_{S^{n-1}}\frac{d\theta}{\abs{\rho\theta-e_n}^\alpha},
\end{align*}
where $\rho := r_1/r_2>1$.

When $\rho\ge 2$, we notice that $\abs{\rho\theta-e_n}\ge \rho/2$, so 
\begin{equation*}
\int\frac{d\theta}{\abs{\rho\theta-e_n}^\alpha}\le C\frac{1}{\rho^\alpha},
\end{equation*}
which implies $C_{r_1,r_2}\lesssim 1/r_1^\alpha$, for $r_1>2r_2$.

When $\rho< 2$, we notice that 
\begin{align*}
\abs{\rho(\theta-e_n)+(\rho-1)e_n}^2 &= 2\rho^2(\theta_n-1)^2 + (\rho-1)^2 + 2\rho(\rho-1)(\theta_n-1) \\
&\ge 2\rho^2(\theta_n-1)^2 + (\rho-1)^2 - a\rho^2(\theta_n-1)^2 - a^{-1}(\rho-1)^2,
\end{align*}
so we can take either $a = 2$ or $a = 1$ to see $\abs{\rho\theta-e_n}\ge c\max\{\rho\abs{\theta-e_n},\rho-1\}$. Hence,
\begin{align*}
\int\frac{d\theta}{\abs{\rho\theta-e_n}^\alpha} &\lesssim \rho^{-\alpha}\int_{\rho\abs{\theta-e_n}>\rho-1}\frac{d\theta}{\abs{\theta-e_n}^\alpha} + (\rho-1)^{-\alpha}\int_{\rho\abs{\theta-e_n}<\rho-1}d\theta \\
&\simeq (\rho-1)^{n-1-\alpha}/\rho^{n-1},
\end{align*}
which implies $C_{r_1,r_2}\lesssim \abs{r_1-r_2}^{n-1-\alpha}(r_1r_2)^{-\frac{n-1}{2}}$, for $r_2<r_1<2r_2$.
\end{proof}

\begin{theorem} \label{thm:h_decay}
If $f\in\Sigma_\delta(\R^n)$, for $0<\delta<1$, then for $\abs{\tau}\ge 1$ it holds
\begin{equation}\label{eq:thm:decay}
\abs{\hat h_\delta[f](\tau)}\le C\norm{f}^2_{\Sigma_\delta}
\begin{cases}
\tau^{-1-2\delta} & \textrm{for } n\ge 3 \textrm{ and } \delta\le \frac{n}{2}-1, \\
\tau^{-\frac{n+2}{4}-\frac{3}{2}\delta}  & \textrm{for } n=2,3 \textrm{ and } \delta> \frac{n}{2}-1,\\
\tau^{-\frac{3}{4}-\frac{3}{2}\delta+} & \textrm{for } n=1.
\end{cases}
\quad \mbox{a.e.}
\end{equation}
The rate of decay is the best possible --- up to the end point when $n = 1$.
\end{theorem}

Theorem~\ref{thm:h_decay} provides an alternative proof of Theorem~\ref{thm:Lipschitz} when $n\ge 3$ and $\delta\le \frac{n}{2}-1$.

\begin{proof}
We can assume that $f\in C_0^\infty(\R^n)$. In fact, for general $f\in\Sigma_\delta(\R^n)$ we can take a sequence of functions $\{f_n\}_n$ in $C_0^\infty(\R^n)$ (Lemma~\ref{thm:smooth_dense}) converging to $f$ in $\Sigma_\delta(\R^n)$. If we reprise the arguments in the proof of \eqref{eq:thm:Fourier_h_L1} we can see that
\begin{equation*}
\norm{\hat h[f]-\hat h[g]}_{L^1(\R\setminus[-1,1])} \le C\norm{f-g}_{\Sigma_\delta}(\norm{f}_{\Sigma_\delta}+\norm{g}_{\Sigma_\delta}).
\end{equation*}
Thus, passing to a sub-sequence $h_\delta[f_n]\to h_\delta[f]$ a.e. and we are done.

Since $\hat h_\delta(-\tau) = \conj{\hat h}_\delta(\tau)$, we assume $\tau>0$. We re-scale \eqref{eq:thm:Fourier_Transform_h} to write
\begin{align*}
\hat{h}_\delta(\tau) &= -2b_{n,\delta}\tau^{\frac{n}{2}-\delta-1}\int \hat{f}(\sqrt{\tau}\xi)\conj{\hat{f}}(\sqrt{\tau}\eta)\delta\Big(1-\frac{\abs{\eta}^2-\abs{\xi}^2}{2}\Big)\frac{d\xi d\eta}{\abs{\xi-\eta}^{n+2\delta}};
\end{align*}
passing to spherical coordinates we have
\begin{align*}
\abs{\hat h_\delta(\tau)} &\le C\tau^{\frac{n}{2}-\delta-1}\int \delta\Big(1-\frac{r_2^2-r_1^2}{2}\Big)r_1^{n-1}r_2^{n-1} \\
&\hspace*{3cm} \Big[\int_{S^{n-1}\times S^{n-1}}\frac{\abs{\hat{f}(\sqrt{\tau}r_1\theta_1)\hat{f}(\sqrt{\tau}r_2\theta_2)}\,d\theta_1d\theta_2}{\abs{r_1\theta_1-r_2\theta_2}^{n+2\delta}}\Big]\,dr_1dr_2.
\end{align*}
The term $\delta(1-(r_2^2-r_1^2)/2)$ forces $\abs{r_2/r_1} = \sqrt{1+2/r_1^2}$.

We apply Lemma~\ref{thm:aux_decay_h} to the term in parentheses to deduce
\begin{align}
\abs{\hat h_\delta(\tau)} &\le C\tau^{\frac{n}{2}-\delta-1}\int\delta\Big(1-\frac{r_2^2-r_1^2}{2}\Big)r_1^{n-1}r_2^{n-1}\norm{\hat f(\sqrt{\tau}r_1\cdot)}_2\norm{\hat f(\sqrt{\tau}r_2\cdot)}_2 \notag \\
&\hspace*{2.8cm}\Big[\ind_{\{r_1>\sqrt{2/3}\}}\frac{1}{(r_1r_2)^{\frac{n}{2}-1-\delta}}+\ind_{\{r_1<\sqrt{2/3}\}}\frac{1}{r_2^{n+2\delta}}\Big]\,dr_1 dr_2 \notag \\
&:= I_{\{r_1>\sqrt{2/3}\}} + I_{\{r_1<\sqrt{2/3}\}}. \label{eq:decay_spherical}
\end{align}

We bound the contribution over the region $\{r_1>\sqrt{2/3}\}$ as
\begin{align}
I_{\{r_1>\sqrt{2/3}\}} &\le C\tau^{\frac{n}{2}-\delta-1}\int_{r_1>\sqrt{2/3}}\delta\Big(1-\frac{r_2^2-r_1^2}{2}\Big)  \notag \\
&\hspace*{2cm} \Big[r_1^{n+2\delta}\norm{\hat f(\sqrt{\tau}r_1\cdot)}_2^2 + r_2^{n+2\delta}\norm{\hat f(\sqrt{\tau}r_2\cdot)}_2^2\Big]\,dr_1dr_2 \notag \\
&\le C\tau^{\frac{n}{2}-\delta-1}\int_{r>c}r^{n-1+2\delta}\norm{\hat f(\sqrt{\tau}r\cdot)}_2^2\,dr \notag \\
&\le C\tau^{-1-2\delta}\norm{\abs{\xi}^\delta\hat f}_2^2.\label{eq:decay_r_big}
\end{align} 
It remains to control the integral over $\{r_1<\sqrt{2/3}\}$.

When $r_1<c$, the term $\delta(1-(r_2^2-r_1^2)/2)$ forces $r_2\simeq 1$, so 
\begin{equation*}
I_{\{r_1<\sqrt{2/3}\}} \le C\tau^{\frac{n}{2}-\delta-1}\int_{r_1<\sqrt{2/3}} r_1^{n-1}\norm{\hat f(\sqrt{\tau}r_1\cdot)}_2\norm{\hat f(\sqrt{\tau}(2+r_1^2)^\frac{1}{2}\cdot)}_2\,dr_1
\end{equation*}
We leave aside momentarily the case $n = 1$. We use H\"older to get (we write $r = r_1$)
\begin{multline*}
I_{\{r_1<\sqrt{2/3}\}} \le C\tau^{\frac{n}{2}-\delta-1}\Big(\int_{r<c}r^{2(n-1)-1}\norm{\hat f(\sqrt{\tau}r\cdot)}_2^2\,dr\Big)^\frac{1}{2} \\
\Big(\int_{r<c}r\norm{\hat f(\sqrt{\tau}(2+r^2)^\frac{1}{2}\cdot)}_2^2\,dr\Big)^\frac{1}{2},
\end{multline*}
after the change of variable $t = \sqrt{2+r^2}$ we get
\begin{align*}
I_{\{r_1<\sqrt{2/3}\}} &\le C\tau^{\frac{n}{2}-\delta-1}\Big(\int_{\abs{\xi}<c}\abs{\xi}^{n-2}\abs{\hat f(\sqrt{\tau}\xi)}^2\,d\xi\Big)^\frac{1}{2}\Big(\int_{\abs{\xi}\simeq 1}\abs{\hat f(\sqrt{\tau}\xi)}^2\,d\xi\Big)^\frac{1}{2} \\
&\le C\tau^{\frac{n}{4}-\frac{3}{2}\delta-1}\Big(\int_{\abs{\xi}<c}\abs{\xi}^{n-2}\abs{\hat f(\sqrt{\tau}\xi)}^2\,d\xi\Big)^\frac{1}{2}\norm{\abs{\xi}^\delta \hat f}_2.
\end{align*}

If $n-2\ge 2\delta$, then we bound the last integral in parentheses as
\begin{equation*}
\int_{\abs{\xi}<c}\abs{\xi}^{n-2}\abs{\hat f(\sqrt{\tau}\xi)}^2\,d\xi \le C\int_{\abs{\xi}<c}\abs{\xi}^{2\delta}\abs{\hat f(\sqrt{\tau}\xi)}^2\,d\xi\le C\tau^{-\frac{n}{2}-\delta}\norm{\abs{\xi}^\delta\hat f}^2_2,
\end{equation*}
so $I_{\{r_1<\sqrt{2/3}\}}\le C\tau^{-2\delta-1}$, which together with \eqref{eq:decay_spherical} and \eqref{eq:decay_r_big} implies the first case in \eqref{eq:thm:decay}. 

If $0\le n-2<2\delta$, then 
\begin{align*}
\int_{\abs{\xi}<c}\abs{\xi}^{n-2}\abs{\hat f(\sqrt{\tau}\xi)}^2\,d\xi &\le \tau^{-n+1}\int_{\abs{\xi}<\sqrt{\tau}c}\abs{\xi}^{n-2}\abs{\hat f(\xi)}^2\,d\xi \\
&\le C\tau^{-n+1}\norm{f}^2_{\Sigma_\delta}.
\end{align*}
Hence $I_{\{r_1<\sqrt{2/3}\}}\le C\tau^{-\frac{n+2}{4}-\frac{3}{2}\delta}$, which together with \eqref{eq:decay_spherical} and \eqref{eq:decay_r_big} implies the second case in \eqref{eq:thm:decay}. 

Now we consider the case $n = 1$, so we have to bound the integral
\begin{equation*}
I'_{\{r<\sqrt{2/3}\}} := C\tau^{-\frac{1}{2}-\delta}\int_{r<c}\abs{\hat f(\sqrt{\tau}r)}\abs{\hat f(\sqrt{\tau}(2+r^2)^\frac{1}{2})}\,dr.
\end{equation*}
We intend to use the embedding $\Sigma_\delta \xhookrightarrow{} L^p$ in Lemma~\ref{thm:embedding_Sigma}. We apply  H\"older inequality (twice) and the change of variables $t = \sqrt{2+r^2}$ to get
\begin{align*}
I'_{\{r<\sqrt{2/3}\}} &\le C\tau^{-\frac{1}{2}-\delta}\Big(\int_{r<c}\abs{\hat f(\sqrt{\tau}r)}^p\,dr\Big)^\frac{1}{p}\Big(\int_{\sqrt{2}}^{\sqrt{2+c^2}}\abs{\hat f(\sqrt{\tau}t)}^{p'}\frac{dt}{\sqrt{t^2-2}}\Big)^\frac{1}{p'} \\
&\le C_\varepsilon \tau^{-\frac{1}{2}-\delta-\frac{1}{2p}}\norm{\hat f}_p\Big(\int_{t\simeq 1}\abs{\hat f(\sqrt{\tau}t)}^{(2+\varepsilon)p'}\,dt\Big)^\frac{1}{(2+\varepsilon)p'}.
\end{align*}

If $\delta< \frac{1}{2}$, then we take $\frac{1}{(2+\varepsilon)p'} = \frac{1}{2}-\delta$, in which case $\frac{1}{p} = (2+\varepsilon)\delta-\varepsilon/2<\frac{1}{2}+\delta$, so that
\begin{equation*}
I'_{\{r<\sqrt{2/3}\}} \le C_\varepsilon \tau^{-\frac{1}{2}-\delta-\frac{1}{2p}-\frac{1}{2(2+)p'}}\norm{\hat f}_p\norm{\hat f}_{(2+)p'}\le C \tau^{-\frac{3}{4}-\frac{3}{2}\delta+}\norm{f}_{\Sigma_\delta}^2,
\end{equation*}
which together with \eqref{eq:decay_spherical} and \eqref{eq:decay_r_big} implies the third case in \eqref{eq:thm:decay}, for $\delta< \frac{1}{2}$.

If $\delta\ge \frac{1}{2}$, then we take $p$ very large and notice that
\begin{equation*}
\Big(\int_{t\simeq 1}\abs{\hat f(\sqrt{\tau}t)}^{(2+\varepsilon)p'}\,dt\Big)^\frac{1}{(2+\varepsilon)p'}\le \Big(\int_{t\simeq 1}\abs{\hat f(\sqrt{\tau}t)}^2\,dt\Big)^\frac{\theta}{2}\Big(\int\abs{\hat f(\sqrt{\tau}t)}^p\,dt\Big)^\frac{1-\theta}{p},
\end{equation*}
where $\frac{1}{(2+)p'} = \frac{\theta}{2}+\frac{1-\theta}{p}$, so $0<\theta <1$ can be made arbitrarily close to 1 if $p\gg 1$. Hence, 
\begin{align*}
I'_{\{r<\sqrt{2/3}\}} &\le C_\varepsilon \tau^{-\frac{1}{2}-\delta-\frac{2-\theta}{2p}}\norm{\hat f}_p^{2-\theta}\Big(\int_{t\simeq 1}\abs{\hat f(\sqrt{\tau}t)}^2\,dt\Big)^\frac{\theta}{2} \\
&\le C_\varepsilon \tau^{-\frac{3}{4}-\frac{3}{2}\delta+}\norm{\hat f}_p^{2-\theta}\norm{\abs{\xi}^\delta \hat f}_2^\theta \\
&\le C \tau^{-\frac{3}{4}-\frac{3}{2}\delta+}\norm{f}_{\Sigma_\delta}^2,
\end{align*}
which together with \eqref{eq:decay_spherical} and \eqref{eq:decay_r_big} concludes the proof of the last case in \eqref{eq:thm:decay}.  
\\[0.2cm]
\indent\textit{Sharpness of the rate of decay}
\\[0.2cm]
The example used in Theorem~\ref{thm:Lipschitz} shows that the decay $\abs{\tau}^{-1-2\delta}$, for $f\in\Sigma_\delta$, cannot be improved, so we turn to the case $n \le 3$.

Let $\zeta\in C^\infty_0(\R)$ be a symmetric cut-off of $B_1$, and let $dS_k$ denote the standard measure on the sphere with radius $2^k$ and center at the origin. To construct the example, we define $\zeta_k(\xi) := 2^{k(n-1)}\zeta(2^k\xi)$ and set 
\begin{equation*}
\hat f(\xi) := \zeta(\xi) + \sum_{k\ge 1} 2^{-k(\frac{n-2}{2}+\delta)}\frac{1}{k^2}(\zeta_k*dS_{k})(\xi).
\end{equation*} 

\begin{center}
\includegraphics[scale=0.8]{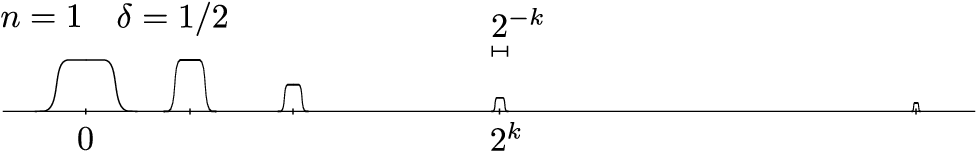}
\end{center}

\noindent Direct computation shows that$\norm{\abs{\xi}^\delta \hat f}_2<\infty$, so we must show that $\norm{\abs{x}^\delta f}_2<\infty$; we only consider the harder case $n\ge 2$.

By the triangle inequality
\begin{equation*}
\norm{\abs{x}^\delta f}_2\le \norm{\abs{x}^\delta \check\zeta}_2 + \sum_{k\ge 1}2^{-k(\frac{n-2}{2}+\delta)}\frac{1}{k^2}\norm{\abs{x}^\delta \check{\zeta}_k(dS_{k})^\vee}_2,
\end{equation*}
After the dilation $x\mapsto 2^{-k} x$, each term in the sum gets into
\begin{equation*}
\norm{\abs{x}^\delta\check{\zeta}_k(dS_{k})^\vee}_2 = 2^{k(\frac{n}{2}-2-\delta)}\norm{\abs{x}^\delta\check{\zeta}(2^{-2k}x) (dS)^\vee}_2,
\end{equation*}
From the inequality $\abs{(dS)^\vee(\xi)}\lesssim \japan{\xi}^{-\frac{n-1}{2}}$ \cite[Ch. VIII-3]{MR1232192} we deduce that $\norm{\abs{x}^\delta\check{\zeta}_k(dS_{k})^\vee}_2 \lesssim 2^{k(\frac{n-2}{2}+\delta)}$, which leads to $\norm{\abs{x}^\delta f}_2<\infty$.

We estimate now $\abs{\hat h_\delta(\tau)}$ for $\tau = 2^{2k-1}$ and $k\gg 1$:
\begin{align*}
\abs{\hat h_\delta(\tau)}&\ge c\frac{1}{k^2}2^{-k(\frac{n-2}{2}+\delta)}\int_{\R^{2n}}\zeta(\xi)(\zeta_k*dS_{k})(\eta)\delta\Big(\tau-\frac{\eta^2-\xi^2}{2}\Big)\frac{d\xi d\eta}{\abs{\xi-\eta}^{n+2\delta}} \\
&\ge c\frac{1}{k^2}2^{-k(\frac{n-2}{2}+\delta)}\int_{\R^2}\zeta(r_1)\zeta(2^k(r_2-\sqrt{2s}))\delta\Big(\tau-\frac{r_2^2-r_1^2}{2}\Big)\frac{r_1^{n-1}dr_1dr_2}{r_2^{1+2\delta}} \\
&\ge c\frac{1}{k^2}2^{-k(\frac{n-2}{2}+\delta)}\tau^{-1-\delta}\int_{\R^2}\zeta(r_1)\zeta(2^k(\sqrt{2s+r_1^2}-\sqrt{2s}))r_1^{n-1}\,dr_1.
\end{align*} 

\begin{center}
\includegraphics[scale=0.9]{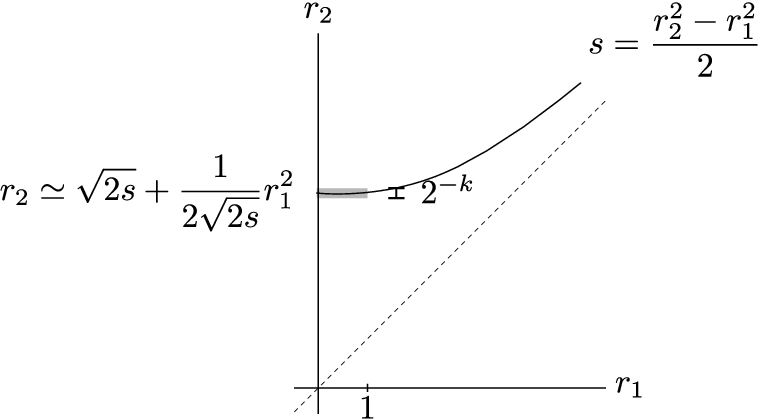}
\end{center}

\noindent Since $\zeta(2^k(\sqrt{2s+r_1^2}-\sqrt{2s}))\gtrsim 1$ for $\abs{\xi}<c$, then

\begin{equation*}
\abs{\hat h_\delta(\tau)} \ge c\frac{1}{k^2}\tau^{-\frac{n+2}{4}-\frac{3}{2}\delta}.
\end{equation*}
Hence, if $\abs{\hat h_\delta(\tau)}\le C\tau^{-\alpha}$, then $\alpha\le \frac{n+2}{4}+\frac{3}{2}\delta$ and the rate of decay in \eqref{eq:thm:decay} cannot be improved.
\end{proof}

% !TEX root = Sigma_Space.tex

\section{Periodic Data}\label{sec:periodic_data}

In the section we extend the definition of $h_\delta$ to solutions of the Schr\"odinger equation with periodic initial data, with the aim to define $h_\delta[f]$ when $f$ is the Dirac comb. 

We choose a real, symmetric function $\psi\in\mathcal{S}(\R^n)$ with $\supp\hat{\psi}\subset B_1$ and $\psi(0) = 1$. Now we approach a periodic function $F$ in $\R^n/\Z^n$ as 
\begin{equation}\label{eq:approach_periodic}
f_\varepsilon(x) := N_\varepsilon^{-1} \psi(\varepsilon x)F(x) = N_\varepsilon^{-1} \psi(\varepsilon x)\sum_{\nu\in\Z^n} \hat{F}(\nu)e(2\pi i x\cdot \nu),
\end{equation}
where $N_\varepsilon^2 = \varepsilon^{-n}\norm{\psi}_2^2\norm{F}_{L^2(\mathbb{T})}^2$ is the normalization constant; henceforth, we will assume that $\norm{F}_{L^2(\mathbb{T})} = 1$. The Fourier transform is
\begin{equation*}
\hat{f}_\varepsilon(\xi) = N_\varepsilon^{-1}\frac{1}{\varepsilon^n}\sum_{\nu\in\Z^n}\hat{F}(\nu)\hat{\psi}((\xi-\nu)/\varepsilon).
\end{equation*}
We want to study how $h_\delta[f_\varepsilon]$ evolves as $\varepsilon\to 0$.

The Fourier transform of $h_\delta[f_\varepsilon]$ away from the origin is
\begin{align*}
\inner{\hat{h}_\delta}{\varphi} &= -2b_{n,\delta}\int_{\R^{2n}}\hat{f}_\varepsilon(\xi)\conj{\hat{f}}_\varepsilon(\eta)\varphi\Big(\frac{\abs{\eta}^2-\abs{\xi}^2}{2}\Big)\frac{d\xi d\eta}{\abs{\xi-\eta}^{n+2\delta}} \\
&= -2b_{n,\delta}\frac{N_\varepsilon^{-2}}{\varepsilon^{2n}}\sum_{\nu_1,\nu_2}\hat{F}(\nu_1)\conj{\hat{F}}(\nu_2) \\
&\hspace*{1.5cm} \int \hat{\psi}((\xi-\nu_1)/\varepsilon)\hat{\psi}((\eta-\nu_2)/\varepsilon)\varphi\Big(\frac{\abs{\xi}^2-\abs{\eta}^2}{2}\Big) \frac{d\xi d\eta}{\abs{\xi-\eta}^{n+2\delta}},
\end{align*}
where $\varphi\in\mathcal{S}(\R)$ is supported away from the origin. In this expression we can distinguish two types of terms: diagonal ($\nu_1= \nu_2$) and off-diagonal ($\nu_1\neq \nu_2$). Diagonal terms are more related to the behavior of $h_\delta$ in the large, and off-diagonal terms are more related to the local phenomena we are interested in.  

\begin{definition}[Decomposition of $h_\delta$]
Let $F$ be a normalized periodic function in $\R^n/\Z^n$. The $\varepsilon$-periodic part $h_{\textrm{p},\varepsilon,\delta}[F]$ (off-diagonal part) is given by
\begin{equation}\label{eq:def:eps_periodic}
\begin{split}
\inner{\hat{h}_{\textrm{p},\varepsilon,\delta}}{\varphi} &:= -\frac{2b_{n,\delta}}{\varepsilon^{2n}\norm{\psi}_2^2}\sum_{\nu_1 \neq \nu_2}\hat{F}(\nu_1)\conj{\hat{F}}(\nu_2) \\
&\hspace*{1cm} \int \hat{\psi}((\xi-\nu_1)/\varepsilon)\hat{\psi}((\eta-\nu_2)/\varepsilon)\varphi\Big(\frac{\abs{\xi}^2-\abs{\eta}^2}{2}\Big) \frac{d\xi d\eta}{\abs{\xi-\eta}^{n+2\delta}},
\end{split}
\end{equation}
where $\varphi\in\mathcal{S}(\R)$ is a test function. The $\varepsilon$-background part $h_{\textrm{b},\varepsilon,\delta}$ (diagonal part) is given by
\begin{equation}\label{eq:def:eps_background}
h_{\textrm{b},\varepsilon,\delta}[F] := h_\delta[f_\varepsilon] - \varepsilon^n h_{\textrm{p},\varepsilon,\delta}[F].
\end{equation}
\end{definition}

Once we have defined the decomposition of $h_\delta$, we concentrate for the moment on the behavior of the $\varepsilon$-periodic part $h_{\textrm{p},\varepsilon,\delta}$ as $\varepsilon$ tends to zero, but first we need a definition.

\begin{definition}
Let $F$ be a normalized periodic function in $\R^n/\Z^n$. The periodic limit $h_{\textrm{p},\delta}[F]$ is given by
\begin{equation}\label{eq:def:periodic}
\begin{split}
\inner{\hat{h}_{\textrm{p},\delta}}{\varphi} &:= -\frac{2b_{n,\delta}}{\norm{\psi}_2^2}\sum_{\nu_1 \neq \nu_2}\hat{F}(\nu_1)\conj{\hat{F}}(\nu_2)\varphi\Big(\frac{\abs{\nu_1}^2-\abs{\nu_2}^2}{2}\Big)\frac{1}{\abs{\nu_1-\nu_2}^{n+2\delta}},
\end{split}
\end{equation}
where $\varphi\in\mathcal{S}(\R)$ is a test function.
\end{definition}

\begin{lemma}\label{thm:periodic_h}
Let $F$ be a normalized periodic function such that $\hat{F}\in \ell^2(\abs{\nu}^{2\delta})$. If $h_{\textrm{p},\varepsilon,\delta}[F]$ and $h_{\textrm{p},\delta}[F]$ are the distributions in \eqref{eq:def:eps_periodic} and \eqref{eq:def:periodic}, respectively, then $h_{\textrm{p},\varepsilon,\delta}[F]$ converges uniformly in compact sets to $h_{\textrm{p},\delta}[F]$, and $\norm{\hat h_{\textrm{p},\delta}[F]}_{L^1}\lesssim 1$.
\end{lemma}

\begin{proof}
The distribution $\hat h_{\textrm{p},\varepsilon,\delta}[F]$ is an integrable function. In fact, we can bound $\abs{\inner{\hat{h}_{\textrm{p},\varepsilon,\delta}}{\varphi}}$ as
\begin{equation*}
\abs{\inner{\hat{h}_{\textrm{p},\varepsilon,\delta}}{\varphi}}\le C\norm{\varphi}_\infty,
\end{equation*}
where $C$ is independent of $\varepsilon$. The same arguments used in Theorem~\ref{thm:h_away_zero} to prove \eqref{eq:thm:Fourier_h_L1} show that $\norm{\hat h_{\textrm{p},\varepsilon,\delta}}_{L^1}\lesssim 1$, so there exists a measure $\mu$ and a sequence $\{h_{\textrm{p},\varepsilon_k,\delta}\}_k$, with $\varepsilon_k\to 0$, that converges weakly$^*$ to $\mu$ with $\abs{\mu}(\R)\lesssim 1$.

To evaluate the integral in \eqref{eq:def:eps_periodic} we fix a number $R\ge 1$ and notice that for $(\xi,\eta)$ at distance less than $\varepsilon$ from $(\nu_1,\nu_2)$ we have two bounds: if $\{\max\abs{\nu_i}> R\}$ then
\begin{multline*}
\varphi\Big(\frac{\abs{\xi}^2-\abs{\eta}^2}{2}\Big)\frac{1}{\abs{\xi-\eta}^{n+2\delta}} - \varphi\Big(\frac{\abs{\nu_1}^2-\abs{\nu_2}^2}{2}\Big)\frac{1}{\abs{\nu_1-\nu_2}^{n+2\delta}} = \\
 \BigO\Big(\frac{\norm{\varphi}_\infty}{\abs{\nu_1-\nu_2}^{n+2\delta}}\Big),
\end{multline*}
and if $\{\max\abs{\nu_i}< R\}$ then
\begin{multline*}
\varphi\Big(\frac{\abs{\xi}^2-\abs{\eta}^2}{2}\Big)\frac{1}{\abs{\xi-\eta}^{n+2\delta}} - \varphi\Big(\frac{\abs{\nu_1}^2-\abs{\nu_2}^2}{2}\Big)\frac{1}{\abs{\nu_1-\nu_2}^{n+2\delta}} = \\
 \BigO\Big(\frac{\norm{\varphi}_\infty\varepsilon}{\abs{\nu_1-\nu_2}^{n+2\delta+1}} + \frac{\norm{\varphi'}_\infty R\varepsilon}{\abs{\nu_1-\nu_2}^{n+2\delta}}\Big).
\end{multline*}
With these two bounds we get the following estimate for \eqref{eq:def:eps_periodic}:
\begin{align*}
\inner{\hat{h}_{\textrm{p},\varepsilon,\delta}}{\varphi} &= -\frac{2b_{n,\delta}}{\norm{\psi}_2^2}\sum_{\nu_1\neq \nu_2}\hat{F}(\nu_1)\conj{\hat{F}}(\nu_2)\Big[\varphi\Big(\frac{\abs{\nu_1}^2-\abs{\nu_2}^2}{2}\Big)\frac{1}{\abs{\nu_1-\nu_2}^{n+2\delta}} + \\
&\hspace*{2.5cm} + \ind_{\{\max\abs{\nu_i}>R\}}\BigO\Big(\frac{\norm{\varphi}_\infty}{\abs{\nu_1-\nu_2}^{n+2\delta}}\Big) +\\
&\hspace*{2.5cm} + \ind_{\{\max\abs{\nu_i}<R\}}\BigO\Big(\frac{\norm{\varphi}_\infty\varepsilon}{\abs{\nu_1-\nu_2}^{n+2\delta+1}} + \frac{\norm{\varphi'}_\infty R\varepsilon}{\abs{\nu_1-\nu_2}^{n+2\delta}}\Big)\Big] \\
&:= \inner{\hat h_{\textrm{p},\delta}}{\varphi} + \textrm{E}_1 + \textrm{E}_2. 
\end{align*}
We bound the first error term as
\begin{align}
\textrm{E}_1 &\le C\norm{\varphi}_\infty\sum_{\nu_1}\abs{\hat{F}(\nu_1)}^2\sum_{\max\abs{\nu_i}>R}\frac{1}{\abs{\nu_1-\nu_2}^{n+2\delta}} \notag \\
&\le C\norm{\varphi}_\infty\sum_{\nu_1}\abs{\hat{F}(\nu_1)}^2\Big(\ind_{\abs{\nu_1}<R/2}R^{-2\delta} + \ind_{\abs{\nu_1}>R/2}\Big) \notag \\
&\le C\norm{\varphi}_\infty R^{-2\delta}(1 + \norm{\hat{F}}_{\ell^2(\abs{\nu}^{2\delta})}^2); \label{eq:limit_periodic_er1}
\end{align}
we bound the second error term as
\begin{equation}\label{eq:limit_periodic_er2}
\textrm{E}_2 \le C(\norm{\varphi}_\infty\varepsilon + \norm{\varphi'}_\infty R\varepsilon).
\end{equation}
We have thus
\begin{multline*}
\inner{\hat{h}_{\textrm{p},\varepsilon,\delta}}{\varphi} = -\frac{2b_{n,\delta}}{\norm{\psi}_2^2}\sum_{\nu_1\neq \nu_2}\hat{F}(\nu_1)\conj{\hat{F}}(\nu_2)\varphi\Big(\frac{\abs{\nu_1}^2-\abs{\nu_2}^2}{2}\Big)\frac{1}{\abs{\nu_1-\nu_2}^{n+2\delta}} + \\
+ \BigO(\norm{\varphi}_\infty R^{-2\delta} + \norm{\varphi}_\infty\varepsilon + \norm{\varphi'}_\infty R\varepsilon).
\end{multline*}
Taking $R = \varepsilon^{-\frac{1}{2}}$ we see that $\inner{\hat{h}_{\textrm{p},\varepsilon,\delta}}{\varphi}\xrightarrow{\varepsilon\to 0} \inner{\hat h_{\textrm{p},\delta}}{\varphi}$ when $\varphi\in \mathcal{S}(\R)$, which implies that $\mu = \hat h_{\textrm{p},\delta}$ is unique and that $\norm{\hat h_{\textrm{p},\delta}}_{L^1}\lesssim 1$. 

To compute $h_{\textrm{p},\delta}$ we set $\varphi(\tau) = e^{2\pi i\tau t}$, and since the error terms \eqref{eq:limit_periodic_er1} and \eqref{eq:limit_periodic_er2} are uniform in $t$ when $\abs{t}\le T$, for any $T>0$, then we conclude that $h_{p,\varepsilon,\delta}$ converges uniformly in compact sets to $h_{\textrm{p},\delta}$.
\end{proof}

The function $h_{\textrm{p},\delta}[F]$ is our desired extension of $h_\delta[f]$ to periodic functions, and we may write it as
\begin{equation}\label{eq:def:periodic_other}
\hat{h}_{p,\delta}[F](\tau) = -\frac{2b_{n,\delta}}{\norm{\psi}_2^2}\sum_{k\in\Z}\delta_\frac{k}{2}(\tau)\sum_{\substack{\nu_1\neq \nu_2 \\ \abs{\nu_1}^2- \abs{\nu_2}^2 = k}}\hat{F}(\nu_1)\conj{\hat{F}}(\nu_2)\frac{1}{\abs{\nu_1-\nu_2}^{n+2\delta}}.
\end{equation}
We observe that $h_{p,\delta}$ is a periodic function with period 2.

Lemma~\ref{thm:periodic_h} says that we can recover $h_{\textrm{p},\delta}[F]$ if we remove $h_{\textrm{b},\varepsilon,\delta}[F]$ from $h_\delta[f_\varepsilon]$, multiply by $\varepsilon^{-n}$ and then take the limit as $\varepsilon\to 0$, \textit{i.e.} we can recover $h_{\textrm{p},\delta}[F]$ if we renormalize $h_\delta[f_\varepsilon]$.

Now we investigate the background of $h_\delta[f_\varepsilon]$, which contains the information of $\hat h_\delta[f_\varepsilon]$ around zero.

\begin{lemma}\label{thm:h_background}
Let $F$ be a normalized periodic function in $\R^n/\Z^n$. If $\hat{F}\in\ell^2(\abs{\nu}^{2\delta})$, then the background $h_{\textrm{b},\varepsilon,\delta}[F]$ is
\begin{align}
h_{\textrm{b},\varepsilon,\delta}(t) &= \frac{\varepsilon^n}{\norm{\psi}_2^2}\sum_\nu\abs{\hat{F}(\nu)}^2\int\abs{x}^{2\delta}\abs{e^{-i\varepsilon^2t\hbar\Delta/2}\psi(\varepsilon(x- t\nu))}^2\,dx, \label{eq:thm:h_background}\\
&= \frac{\varepsilon^{-2\delta}}{\norm{\psi}_2^2}\int\abs{x}^{2\delta}\abs{\psi(x)}^2\,dx + o(1),\label{eq:thm:background_limit}
\end{align}
where the error term $o(1)$ is uniform in compact sets.

If $\hat F\in \ell^2(\abs{\nu}^{n+2\delta})$, then
\begin{equation}\label{eq:thm:Background_HF}
\norm{P_{>\frac{1}{4}}h_{\textrm{b},\varepsilon,\delta}}_\infty = o(\varepsilon^n).
\end{equation}
If we assume further that $n = 1$ and $\delta<\frac{1}{2}$, then
\begin{equation}\label{eq:thm:background_n_1}
h_{\textrm{b},\varepsilon,\delta}(t) = \frac{\varepsilon^{-2\delta}}{\norm{\psi}_2^2}\int\abs{x}^{2\delta}\abs{\psi(x)}^2\,dx + o(\varepsilon),
\end{equation}
where the error term $o(\varepsilon)$ is uniform in compact sets.
\end{lemma}

We can also understand \eqref{eq:thm:h_background} as a self-interaction term since the evolution of $e^{it\hbar\Delta/2}f_\varepsilon$ is
\begin{equation*}
e^{it\hbar\Delta/2}f_\varepsilon = N_\varepsilon^{-1}\sum_{\nu\in\Z^n}\hat F(\nu)e^{2\pi ix\cdot\nu - \pi it\abs{\nu}^2} e^{i\varepsilon^2t\hbar\Delta/2}\psi(\varepsilon(x-t\nu)).
\end{equation*}
Hence, $h_{\textrm{p},\delta}[F]$ represents the sum of the pairwise interaction of different waves.

Equation \eqref{eq:thm:background_limit} says that $\varepsilon^{2\delta}h_{\textrm{b},\varepsilon,\delta}$ tends to a constant function as $\varepsilon\to 0$. Unfortunately, the rate of convergence is not fast enough, so the $\varepsilon$-periodic part $\varepsilon^n h_{\textrm{p},\varepsilon,\delta}$ may be thwarted by the noise in the limit. However, if $F$ is smooth, \textit{i.e.} $\hat F\in \ell^2(\abs{\nu}^{n+2\delta})$, then the high frequencies $P_{>\frac{1}{4}}h_{\textrm{b},\varepsilon,\delta}$ are smaller than $\varepsilon^n h_{\textrm{p},\delta}$ and, in the limit, we can think of $h_\delta[f_\varepsilon]$ as
\begin{equation*}
h_\delta[f_\varepsilon]\approx P_{<\frac{1}{4}}h_{\textrm{b},\varepsilon,\delta} + \varepsilon^n h_{\textrm{p},\delta},
\end{equation*}
where $P_{<\frac{1}{4}}h_{\textrm{b},\varepsilon,\delta}$ is an analytic function which remains essentially constant at scale $2$, while $h_{\textrm{p},\delta}$ is periodic with period 2. This representation offers the possibility of ``watching'' $h_{\textrm{p},\delta}$ numerically as tiny oscillations over a smooth background.

\begin{proof}[Proof of Lemma~\ref{thm:h_background}]
We begin with the proof of \eqref{eq:thm:h_background}. Since $\supp\hat{\psi}\subset B_1$, we can write $h_\delta[f_\varepsilon]$ as
\begin{align*}
h_\delta(t) &= b_{n,\delta}\int\frac{\abs{e^{-\pi it\abs{\xi}^2}\hat{f}_\varepsilon(\xi)-e^{-\pi it\abs{\eta}^2}\hat{f}_\varepsilon(\eta)}^2}{\abs{\xi-\eta}^{n+2\delta}}\,d\xi d\eta \\
&= \frac{b_{n,\delta}}{\varepsilon^n\norm{\psi}_2^2}\sum_{\nu}\abs{\hat{F}(\nu)}^2 \\
&\hspace*{0.5cm} \int \Big|e^{-\pi it\abs{\xi}^2}\hat{\psi}\Big(\frac{\xi-\nu}{\varepsilon}\Big)-e^{-\pi it\abs{\eta}^2}\hat{\psi}\Big(\frac{\eta-\nu}{\varepsilon}\Big)\Big|^2\frac{d\xi d\eta}{\abs{\xi-\eta}^{n+2\delta}} + \varepsilon^n h_{p,\varepsilon,\delta}(t) \\
&= \frac{1}{\varepsilon^n\norm{\psi}_2^2}\sum_\nu\abs{\hat{F}(\nu)}^2\int\abs{x}^{2\delta}\abs{e^{-it\hbar\Delta/2}(e^{2\pi i\nu\cdot y}\psi_\varepsilon)(x)}^2\,dx + \varepsilon^n h_{p,\varepsilon,\delta}(t),
\end{align*}
where $e^{2\pi i\nu\cdot y}\psi_\varepsilon(y) := \varepsilon^n e^{2\pi i\nu\cdot y}\psi(\varepsilon y)$, so
\begin{equation*}
\abs{e^{-it\hbar\Delta/2}(e^{2\pi i\nu\cdot y}\psi_\varepsilon)(x)} = \varepsilon^n\abs{(e^{-i\varepsilon^2 t\hbar\Delta/2}\psi)(\varepsilon(x-t\nu))}.
\end{equation*}
We replace it above to get
\begin{equation*}
h_\delta(t) = \frac{\varepsilon^n}{\norm{\psi}_2^2}\sum_\nu\abs{\hat{F}(\nu)}^2\int\abs{x}^{2\delta}\abs{(e^{-i\varepsilon^2t\hbar\Delta/2}\psi)(\varepsilon(x- t\nu))}^2\,dx + \varepsilon^n h_{p,\varepsilon,\delta}(t),
\end{equation*}
which implies \eqref{eq:thm:h_background} by the definition of $h_{\textrm{b},\varepsilon,\delta}$; see \eqref{eq:def:eps_background}.

Let us define
\begin{align*}
A_{\varepsilon,\nu}(t) := b_{n,\delta}\int \abs{e^{-\pi i\varepsilon^2t\abs{\xi}^2}\hat{\psi}(\xi-\varepsilon^{-1}\nu)-e^{-\pi i\varepsilon^2 t\abs{\eta}^2}\hat{\psi}(\eta-\varepsilon^{-1}\nu)}^2\frac{d\xi d\eta}{\abs{\xi-\eta}^{n+2\delta}}
\end{align*}
so that 
\begin{equation}\label{eq:background_short}
h_{\textrm{b},\varepsilon,\delta}(t) = \frac{\varepsilon^{-2\delta}}{\norm{\psi}^2_2}\sum_\nu\abs{\hat F(\nu)}^2A_{\varepsilon,\nu}(t).
\end{equation}
Since $\psi$ is real and symmetric, the Fourier transform of $A_{\varepsilon,\nu}$ is symmetric and then we can restrict ourselves to symmetric test functions, so 
\begin{align}
\inner{\hat A_{\varepsilon,\nu}}{\varphi} &= b_{n,\delta}\int\Big[\abs{\hat \psi(\xi-\varepsilon^{-1}\nu)}^2\varphi(0) + \abs{\hat \psi(\eta-\varepsilon^{-1}\nu)}^2\varphi(0) - \notag \\
&\hspace*{2cm} - 2\varphi\Big(\varepsilon^2\frac{\abs{\eta}^2-\abs{\xi}^2}{2}\Big)\hat{\psi}(\xi-\varepsilon^{-1}\nu)\hat{\psi}(\eta-\varepsilon^{-1}\nu)\Big]\frac{d\xi d\eta}{\abs{\xi-\eta}^{n+2\delta}} \notag \\
&= \varphi(0)\int\abs{x}^{2\delta}\abs{\psi}^2\,dx + \notag \\ 
&\hspace*{0.5cm} + 2b_{n,\delta}\int\Big[\varphi(0)-\varphi\Big(\big(\varepsilon\frac{\eta+\xi}{2} + \nu\big)\cdot\varepsilon(\eta-\xi)\Big)\Big]\hat{\psi}(\xi)\hat{\psi}(\eta)\frac{d\xi d\eta}{\abs{\xi-\eta}^{n+2\delta}} \label{eq:A_nu_fourier}
\end{align}

We use a test function supported in $\R\setminus[-a,a]$ to bound $\hat A_{\varepsilon,\nu}$ away from the origin as
\begin{equation*}
\abs{\inner{\hat A_{\varepsilon,\nu}}{\varphi}}\le C\int\abs{\hat{\psi}(\xi)\hat{\psi}(\eta)\varphi\Big(\big(\varepsilon\frac{\eta+\xi}{2}+\nu\big)\cdot\varepsilon(\eta-\xi)\Big)}\frac{d\xi d\eta}{\abs{\xi-\eta}^{n+2\delta}}.
\end{equation*}
If $\abs{\nu}\le a\varepsilon^{-1}/4$ then $\inner{\hat A_{\varepsilon,\nu}}{\varphi} = 0$ because $\varphi(t) = 0$ when $\abs{t}\le a$. Otherwise, we change variables and bound the integral as
\begin{align*}
\abs{\inner{\hat A_{\varepsilon,\nu}}{\varphi}} &\le C\int\abs{\hat{\psi}(\xi)\hat{\psi}(\eta)\varphi\Big(\big(\varepsilon u +\nu\big)\cdot\varepsilon v\Big)}\frac{du dv}{\abs{v}^{n+2\delta}} \\
&\le C\norm{\varphi}_\infty\int_{\abs{v}>a/(2\abs{\nu}\varepsilon)}\frac{dv}{\abs{v}^{n+2\delta}} \\
&\le C\norm{\varphi}_\infty\Big(\frac{\abs{\nu}\varepsilon}{a}\Big)^{2\delta}.
\end{align*}
Hence, by Hausdorff-Young inequality $\norm{P_{>a}A_{\varepsilon,\nu}}_\infty \lesssim_a (\abs{\nu}\varepsilon)^{2\delta}$ and then, by \eqref{eq:background_short},
\begin{equation*}
\norm{P_{>a}h_{\textrm{b},\varepsilon,\delta}}_\infty \lesssim_a \sum_{\abs{\nu}>a\varepsilon^{-1}/2}\abs{\hat F(\nu)}^2\abs{\nu}^{2\delta} = o_a(1);
\end{equation*}
to prove \eqref{eq:thm:background_limit} it remains to estimate $P_{<\frac{1}{4}}h_{\textrm{b},\varepsilon,\delta}$. If we assume further that $\hat F\in \ell^2(\abs{\nu}^{n+2\delta})$, then we can state the stronger upper bound$\norm{P_{>a}h_{\textrm{b},\varepsilon,\delta}}_\infty = o_a(\varepsilon^n)$, which is \eqref{eq:thm:Background_HF}.

We turn now to the term $P_{<\frac{1}{4}}h_{\textrm{b},\varepsilon,\delta}$. We will prove that 
\begin{equation}\label{eq:background_low_limit}
\inner{\hat h_{\textrm{b},\varepsilon,\delta}}{\varphi} - \varphi(0)\varepsilon^{-2\delta}\norm{\psi}_2^{-2}\int\abs{x}^{2\delta}\abs{\psi}^2\,dx \to 0,
\end{equation}
which implies \eqref{eq:thm:background_limit} after replacing $\varphi$ by the test function $\tau\mapsto \psi(\tau)\cos(2\pi t\tau)$, where $\psi$ is a symmetric cut-off of $[-1/4,1/4]$; the bounds will be uniform in $t$ if $\abs{t}\le T$, so the convergence is uniform in compact sets.

From \eqref{eq:A_nu_fourier} we see that $\varepsilon^{-2\delta}(\inner{\hat A_{\varepsilon,\nu}}{\varphi} - \varphi(0)\int\abs{x}^{2\delta}\abs{\psi}^2\,dx)\to 0$ as $\varepsilon\to 0$. In fact,
\begin{align*}
\varepsilon^{-2\delta}\abs{I_{\varepsilon,\nu}} &:= \varepsilon^{-2\delta}\abs{\inner{\hat A_{\varepsilon,\nu}}{\varphi} - \varphi(0)\int\abs{x}^{2\delta}\abs{\psi}^2\,dx} \\
&\le C\varepsilon^{-2\delta}\int_{\abs{u},\abs{v}<1}\abs{\varphi(0)-\varphi\Big(\big(\varepsilon \frac{\eta+\xi}{2}+\nu\big)\cdot\varepsilon (\eta-\xi)\Big)}\frac{d\xi d\eta}{\abs{\eta-\xi}^{n+2\delta}} \\
&\le C\varepsilon^{2-2\delta}\norm{\varphi''}_\infty
\end{align*}
and the last term tends to zero, so the claim follows. 

To prove \eqref{eq:background_low_limit}, and so \eqref{eq:thm:background_limit}, it suffices to show that $\varepsilon^{-2\delta}\abs{\hat F(\nu)}^2\abs{I_{\varepsilon,\nu}}$ is uniformly dominated in $\varepsilon$ by an integrable (summable) function; recall \eqref{eq:background_short} and $\norm{F}_2 = 1$. To control $I_{\varepsilon,\nu}$, we change variables and bound the integral as
\begin{align*}
\abs{I_{\varepsilon,\nu}} &\le C\int_{\abs{u},\abs{v}<1}\abs{\varphi(0)-\varphi((\varepsilon u+\nu)\cdot\varepsilon v)}\frac{du dv}{\abs{v}^{n+2\delta}} \\
&\lesssim \norm{\varphi''}_\infty\varepsilon^2\japan{\nu}^2\int_{\abs{v}<r}\frac{dv}{\abs{v}^{n-2+2\delta}} + \norm{\varphi}_\infty\int_{r<\abs{v}<1}\frac{dv}{\abs{v}^{n+2\delta}} \\
&\lesssim \norm{\varphi''}_\infty\varepsilon^2\japan{\nu}^2 \min\{r^{2(1-\delta)},1\} + \norm{\varphi}_\infty r^{-2\delta}\ind_{r<1}.
\end{align*}
When $(\varepsilon\japan{\nu})^2<\norm{\varphi}_\infty/\norm{\varphi''}_\infty$ we choose $r=1$ and we get
\begin{equation*}
\abs{I_{\varepsilon,\nu}}\le C\norm{\varphi''}_\infty\varepsilon^2\japan{\nu}^2\le C(\varepsilon\japan{\nu})^{2\delta}\norm{\varphi}_\infty^{1-\delta}\norm{\varphi''}_\infty^\delta.
\end{equation*}
When $(\varepsilon\japan{\nu})^2>\norm{\varphi}_\infty/\norm{\varphi''}_\infty$ we choose $r^2 = \norm{\varphi}_\infty/(\varepsilon^2\japan{\nu}^2\norm{\varphi''}_\infty)$ and we get
\begin{equation*}
\abs{I_{\varepsilon,\nu}}\le C(\varepsilon\japan{\nu})^{2\delta}\norm{\varphi}_\infty^{1-\delta}\norm{\varphi''}_\infty^\delta.
\end{equation*}
Thus, we have that $\varepsilon^{-2\delta}\abs{\hat F(\nu)}^2\abs{I_{\varepsilon,\nu}}\le C\norm{\varphi}_\infty^{1-\delta}\norm{\varphi''}_\infty^\delta\abs{\hat F(\nu)}^2\japan{\nu}^{2\delta}$, and \eqref{eq:background_low_limit} follows by dominated convergence.

The proof of \eqref{eq:thm:background_n_1} goes along the same lines, but the new hypotheses are $n = 1$, $\delta<\frac{1}{2}$ and $\norm{\abs{\nu}^{\frac{1}{2}+\delta}\hat F}_{\ell^2}<\infty$. Since $\delta<\frac{1}{2}$ we have that $\varepsilon^{-1-2\delta}(\inner{\hat A_{\varepsilon,\nu}}{\varphi} - \varphi(0)\int\abs{x}^{2\delta}\abs{\psi}^2\,dx)\to 0$ as $\varepsilon\to 0$, so it suffices to show that $\varepsilon^{-1-2\delta}\abs{\hat F(\nu)}^2\abs{I_{\varepsilon,\nu}}$ is uniformly dominated in $\varepsilon$ by an integrable function.

The previous bounds of $\abs{I_{\varepsilon,\nu}}$ lead to
\begin{equation*}
\abs{I_{\varepsilon,\nu}}\le C\norm{\varphi''}_\infty\varepsilon^2\japan{\nu}^2\le C(\varepsilon\japan{\nu})^{1+2\delta}\norm{\varphi}_\infty^{\frac{1}{2}-\delta}\norm{\varphi''}_\infty^{\frac{1}{2}+\delta},
\end{equation*}
when $(\varepsilon\japan{\nu})^2<\norm{\varphi}_\infty/\norm{\varphi''}_\infty$, and
\begin{equation*}
\abs{I_{\varepsilon,\nu}}\le C(\varepsilon\japan{\nu})^{2\delta}\norm{\varphi}_\infty^{1-\delta}\norm{\varphi''}_\infty^\delta \le C(\varepsilon\japan{\nu})^{1+2\delta}\norm{\varphi}_\infty^{\frac{1}{2}-\delta}\norm{\varphi''}_\infty^{\frac{1}{2}+\delta},
\end{equation*}
when $(\varepsilon\japan{\nu})^2>\norm{\varphi}_\infty/\norm{\varphi''}_\infty$. Therefore, 
\begin{equation*}
\varepsilon^{-1-2\delta}\abs{\hat F(\nu)}^2\abs{I_{\varepsilon,\nu}}\le C\japan{\nu}^{1+2\delta}\norm{\varphi}_\infty^{\frac{1}{2}-\delta}\norm{\varphi''}_\infty^{\frac{1}{2}+\delta}.
\end{equation*}
By dominated convergence again we get \eqref{eq:thm:background_n_1}.
\end{proof}

Recall that our main interest is the Talbot effect in $n = 1$, so \eqref{eq:thm:background_n_1} in Lemma~\ref{thm:h_background} provides the convenient asymptotic representation
\begin{equation*}
h_\delta(t) = \frac{\varepsilon^{-2\delta}}{\norm{\psi}_2^2}\int\abs{x}^{2\delta}\abs{\psi(x)}^2\,dx +\varepsilon h_{p,\delta}(t) + o(\varepsilon),
\end{equation*}
as long as $\hat F\in \ell^2(\abs{\nu}^{1+2\delta})$ and $\delta<\frac{1}{2}$.

We summarize our main findings in the following theorem.

\begin{theorem}
Let $F$ be a normalized periodic function with period 1 in $\R^n$---recall the definition of $f_\varepsilon$ in \eqref{eq:approach_periodic}. 

If $\hat F\in \ell^2(\abs{\nu}^{2\delta})$, then
\begin{equation}
h_\delta[f_\varepsilon](t) = \frac{\varepsilon^{-2\delta}}{\norm{\psi}_2^2}\int\abs{x}^{2\delta}\abs{\psi(x)}^2\,dx + o(1).
\end{equation}

If $\hat F\in \ell^2(\abs{\nu}^{n+2\delta})$, then
\begin{equation}\label{eq:thm:asymptotic_Smooth_F}
h_\delta[f_\varepsilon](t) = P_{<\frac{1}{4}}h_{\textrm{b},\varepsilon,\delta}(t) + \varepsilon^n h_{\textrm{p},\delta}(t) + o(\varepsilon^n).
\end{equation}

If $n = 1$, $\delta<\frac{1}{2}$ and $\hat F\in \ell^2(\abs{\nu}^{1+2\delta})$, then
\begin{equation}\label{eq:thm:asymptotic_Smooth_F_1D}
h_\delta[f_\varepsilon](t) = \frac{\varepsilon^{-2\delta}}{\norm{\psi}_2^2}\int\abs{x}^{2\delta}\abs{\psi(x)}^2\,dx +\varepsilon h_{p,\delta}(t) + o(\varepsilon).
\end{equation}
The error terms in all the limits are uniform in compact sets of $\R$.
\end{theorem}

In Figure~\ref{fig:Asymptotic} we saw how convenient is \eqref{eq:thm:asymptotic_Smooth_F_1D} to visualize $h_{\textrm{p},\delta}$ numerically.

\subsection{The Dirac comb}\label{sec:Dirac_Comb}

Now that we have succeeded in defining a functional $h_{\textrm{p},\delta}[F]$ for a periodic function $F$, we want to pass again to the limit to study the Dirac comb, \textit{i.e.} the periodic distribution $F_D(x) := \sum_{m\in\Z} \delta(x-m)$ in $\R$.

To approach the Dirac comb in $\R$ we use the function
\begin{equation*}
F_{\varepsilon_1} := \sum_{m\in\Z}\varepsilon_1^{-1}e^{-\pi ((x-m)/\varepsilon_1)^2} = \sum_{m\in\Z} e^{-\pi(\varepsilon_1m)^2}e^{2\pi i x m},
\end{equation*} 
and define so the approximation 
\begin{equation*}
f_{\varepsilon_1,\varepsilon_2}(x) := N_{\varepsilon_2}^{-1} \psi(\varepsilon_2 x)\norm{F_{\varepsilon_1}}_2^{-1}F_{\varepsilon_1},
\end{equation*}
where $N_{\varepsilon_2}$ is the normalization constant of $f_{\varepsilon_1,\varepsilon_2}$. Since the periodic function $F_{\varepsilon_1}$, for $\varepsilon_1$ fixed, is smooth, then from \eqref{eq:thm:asymptotic_Smooth_F} we see that $h_\delta[f_{\varepsilon_1,\varepsilon_2}]$ splits into a smooth background and an oscillating, periodic function $h_{\textrm{p},\delta}[F_{\varepsilon_1}]$ when $\varepsilon_2\to 0$. 

%In Figure~\ref{fig:Asymptotic} we can see how $h_\delta[f_{\varepsilon_1,\varepsilon_2}]$, after renormalization, approaches $h_{\textrm{p},\delta}[F_{\varepsilon_1}]$; in this case it is easy to remove the background because it is a constant, as it was proved in \eqref{eq:thm:asymptotic_Smooth_F_1D}.

We use \eqref{eq:def:periodic}, or \eqref{eq:def:periodic_other}, to see that
\begin{equation}\label{eq:periodic_Dirac_Comb}
\hat h_{\textrm{p},\delta}[F_{\varepsilon_1}](\tau) = -\frac{2b_{1,\delta}}{\norm{\psi}_2^2}\sum_{k\in\Z}\delta_\frac{k}{2}(\tau)\sum_{\substack{m_1\neq m_2 \\ \abs{m_1}^2- \abs{m_2}^2 = k}}\hat{F}_{\varepsilon_1}(m_1)\hat{F}_{\varepsilon_1}(m_2)\frac{1}{\abs{m_1-m_2}^{n+2\delta}}.
\end{equation}
At this stage, we let $\varepsilon_1$ go to zero and take the weak limit of $\hat h_{\textrm{p},\delta}[F_{\varepsilon_1}]$ to get the distribution
\begin{equation*}
\hat{h}_{\textrm{p},\delta}[F_D](\tau) := -\frac{2b_{1,\delta}}{\norm{\psi}_2^2}\sum_k\delta_\frac{k}{2}(\tau)\sum_{\substack{m_1\neq m_2 \\ m^2_1-m^2_2 = k}}\frac{1}{\abs{m_1-m_2}^{1+2\delta}},
\end{equation*}
which is our definition of periodic $h_\delta$ for the Dirac comb $F_D$. Surprisingly, $h_{\textrm{p},\delta}[F_D]$ is a pure point measure; to see this, we have to compute the coefficients of $\hat{h}_{\textrm{p},\delta}[F_D]$.

\begin{lemma}
\begin{equation}
\sum_{\substack{m_1\neq m_2 \\ m^2_1-m^2_2 = k}}\frac{1}{\abs{m_1-m_2}^{1+2\delta}} = \begin{dcases}
2\sum_{\substack{d\mid k \\ d>0}}\frac{1}{d^{1+2\delta}} & \textrm{for } k\in\Z \textrm{ odd} \\
\frac{1}{2^{2\delta}}\sum_{\substack{4d\mid k \\ d>0}}\frac{1}{d^{1+2\delta}} & \textrm{for } k\equiv 0\,\Mod 4 \\
0 & \textrm{for } k\equiv 2\,\Mod 4
\end{dcases}
\end{equation}
\end{lemma}
\textit{Remark.} In number theory notation, for $k$ odd the coefficients are $2\sigma_{-1-2\delta}(k)$, and for $k\equiv 0\,\Mod 4$ the coefficients are $2^{-2\delta}\sigma_{-1-2\delta}(k/4)$.
\begin{proof}
We write $m_1^2-m_2^2 = (m_1-m_2)(m_1+m_2) := de = k$, so necessarily $d\mid k$. On the other hand, we have $m_1 = \frac{1}{2}(e+d)$ and $m_2 = \frac{1}{2}(e-d)$, so $d$ and $e$ have the same parity, \textit{i.e.} $d\equiv e\, \Mod 2$. Consequently, 
\begin{equation*}
\sum_{\substack{m_1\neq m_2 \\ m^2_1-m^2_2 = k}}\frac{1}{\abs{m_1-m_2}^{1+2\delta}} = 2\sum_{\substack{d\equiv e \,\Mod 2 \\ de = k, d>0}}\frac{1}{d^{1+2\delta}},
\end{equation*}
from which the Lemma follows.
\end{proof}

\begin{customthm}{\ref{thm:limit_FD}}
{\it
\begin{align}
\begin{split}\label{eq:thm:h_Talbot}
h_{\textrm{p},\delta}[F_D](2t) &= -\frac{2b_{1,\delta}}{\norm{\psi}_2^2}\zeta(2(1+\delta))\Big[\sum_{\substack{(p,q)=1 \\ q>0\textrm{ odd}}}\frac{1}{q^{2(1+\delta)}}\delta_\frac{p}{q}(t) - \\
&\hspace*{0.5cm}- \sum_{\substack{(p,q)=1 \\ q\equiv 2\,\Mod 4}}\frac{2(2^{1+2\delta}-1)}{q^{2(1+\delta)}}\delta_\frac{p}{q}(t) + \sum_{\substack{(p,q)=1 \\ q\equiv 0\,\Mod 4}}\frac{2^{2(1+\delta)}}{q^{2(1+\delta)}}\delta_\frac{p}{q}(t)\Big],
\end{split} \\
&:= \sum_{\substack{(p,q)=1 \\ q>0}}\frac{a_{\delta,q}}{q^{2(1+\delta)}}\delta_\frac{p}{q}(t) \label{eq:def:a_h_Talbot}
\end{align}
where $\zeta(z)$ is the Riemann zeta function.
}
\end{customthm}
\begin{proof}
We split $\hat{h}_{\textrm{p},\delta}[F_D]$ into 
\begin{align*}
\hat{h}_{\textrm{p},\delta}^\textrm{odd}(\tau) &:= 2\sum_{k\textrm{ odd}}\sigma_{-1-2\delta}(k)\delta_\frac{k}{2}(\tau) \\ 
\hat{h}_{\textrm{p},\delta}^\textrm{even}(\tau) &:= \frac{1}{2^{2\delta}}\sum_{k\equiv 0\,\Mod 4}\sigma_{-1-2\delta}(k/4)\delta_\frac{k}{2}(\tau).
\end{align*}
We rearrange the terms in the sum of the odd part so that
\begin{equation*}
\hat{h}_{\textrm{p},\delta}^\textrm{odd}(\tau) = 2\sum_{k\textrm{ odd}}\Big(\sum_{d\mid k} \frac{1}{d^{1+2\delta}}\Big)\delta_\frac{k}{2}(\tau) = 2\sum_{d>0 \textrm{ odd}}\frac{1}{d^{1+2\delta}}\sum_{l\textrm{ odd}}\delta_{\frac{dl}{2}}(\tau).
\end{equation*}
The very last sum is a Dirac comb supported on the arithmetic progression $\{l\textrm{ odd}\mid dl/2\}$, so the inverse Fourier transform of $\hat{h}_{\textrm{p},\delta}^\textrm{odd}$ is
\begin{align*}
h_{\textrm{p},\delta}^\textrm{odd}(t) &= 2\sum_{d>0\textrm{ odd}}\frac{1}{d^{2(1+\delta)}}\sum_{l\in\Z}(-1)^l\delta_\frac{l}{d}(t) \\
&= 2\sum_{\substack{(p,q)=1 \\ q>0}}\delta_\frac{p}{q}(t)\sum_{\substack{d>0\textrm{ odd}, l \\ l/d = p/q}} \frac{(-1)^l}{d^{2(1+\delta)}}.
\end{align*}
Since $q\mid d$ and $p\mid l$, then
\begin{equation*}
h_{\textrm{p},\delta}^\textrm{odd}(t) = 2\sum_{\substack{(p,q)=1 \\ q>0\textrm{ odd}}}\delta_\frac{p}{q}(t)\frac{(-1)^p}{q^{2(1+\delta)}}\sum_{r>0\textrm{ odd}}\frac{1}{r^{2(1+\delta)}}.
\end{equation*}
We follow a similar argument to evaluate the even part
\begin{align*}
\hat{h}_{\textrm{p},\delta}^\textrm{even}(\tau) &= \frac{1}{2^{2\delta}}\sum_{k\in\Z,\,d\mid k}\delta_{2k}(\tau)\frac{1}{d^{1+2\delta}} = \frac{1}{2^{2\delta}}\sum_{d>0}\frac{1}{d^{1+2\delta}}\sum_{l\in\Z}\delta_{2dl}(\tau);
\end{align*}
hence, the inverse Fourier transform of $\hat{h}_{\textrm{p},\delta}^\textrm{even}$ is 
\begin{align*}
h_{\textrm{p},\delta}^\textrm{even}(t) &= \frac{1}{2^{1+2\delta}}\sum_{d>0}\frac{1}{d^{2(1+\delta)}}\sum_{l\in\Z}\delta_\frac{l}{2d}(t) \\
&= 2\sum_{\substack{(p,q)=1 \\ q>0}}\delta_\frac{p}{q}(t)\sum_{\substack{d>0\textrm{ even},l \\ l/d = p/q}}\frac{1}{d^{2(1+\delta)}} \\
&= 2\Big[\sum_{\substack{(p,q)=1 \\ q>0\textrm{ odd}}} \delta_\frac{p}{q}(t)\frac{1}{q^{2(1+\delta)}}\sum_{r>0\textrm{ even}}\frac{1}{r^{2(1+\delta)}} + \\
&\hspace*{5cm} +\sum_{\substack{(p,q)=1 \\ q>0\textrm{ even}}} \delta_\frac{p}{q}(t)\frac{1}{q^{2(1+\delta)}}\sum_{r>0}\frac{1}{r^{2(1+\delta)}}\Big].
\end{align*}
We sum $h_{\textrm{p},\delta}^\textrm{even}$ and $h_{\textrm{p},\delta}^\textrm{odd}$ to conclude that
\begin{multline*}
h_{\textrm{p},\delta}(t) = -\frac{4b_{1,\delta}}{\norm{\psi}^2_2}\Big[\sum_{\substack{p\textrm{ even} \\ q>0\textrm{ odd}}}\frac{\zeta(2(1+\delta))}{q^{2(1+\delta)}}\delta_\frac{p}{q}(t) + \sum_{\substack{(p,q)=1 \\ q>0,p\textrm{ odds}}}\frac{-\eta(2(1+\delta))}{q^{2(1+\delta)}}\delta_\frac{p}{q}(t) + \\ 
+\sum_{\substack{(p,q)=1 \\ q>0\textrm{ even}}}\frac{\zeta(2(1+\delta))}{q^{2(1+\delta)}}\delta_\frac{p}{q}(t)\Big].
\end{multline*} 
where $\eta(z) = -\sum_{n>0}(-1)^n/n^z = (1-2^{1-z})\zeta(z)$ is the Dirichlet eta function. Finally, we dilate and rearrange the terms so as to get \eqref{eq:thm:h_Talbot}.
\end{proof}

%In Figure \ref{fig:Asymptotic} the limiting function $h_{\textrm{p},\delta}[F_{\varepsilon_1}]$, where $\varepsilon_1 = 0.2$ is fixed and $\varepsilon_2$ tends to zero, does not resemble something like \eqref{eq:thm:h_Talbot} because $\varepsilon_1$ is not sufficiently small. It is very demanding computationally to fix a sufficiently small $\varepsilon_1$ so as to see the emergence of \eqref{eq:thm:h_Talbot} by letting $\varepsilon_2$ go to zero. However, we can plot $h_{\textrm{p},\delta}[F_{\varepsilon_1}]$ directly using  \eqref{eq:periodic_Dirac_Comb} with $\varepsilon_1$ very small (Figure~\ref{fig:Dirac_deltas}), which shows how the peaks develop at rational times as $\varepsilon_1$ tends to zero.

As a side remark, we notice that in higher dimensions the limit of \eqref{eq:periodic_Dirac_Comb} as $\varepsilon_1\to 0^+$ does not exit. In the limit, the Fourier coefficients would be 
\begin{equation*}
r_k = C_\delta\sum_{\substack{m_1\neq m_2 \\ \abs{m_1}^2- \abs{m_2}^2 = k}}\frac{1}{\abs{m_1-m_2}^{n+2\delta}}. \end{equation*}
We define the new variables $l_1 = m_1 - m_2$ and $l_2 = m_1 + m_2$ so that
\begin{equation*}
r_k = C_\delta\sum_{\substack{l_1\neq 0,\,l_1\equiv l_2\Mod{2} \\ l_1\cdot l_2 = k}} \frac{1}{\abs{l_1}^{n+2\delta}}.
\end{equation*}
For $l_1$ fixed, since $l_1\cdot l_2 = k$ has either no solution or infinitely many, it follows that $r_k = \infty$ for every integer $k$. 

To study the function $h_{\textrm{p},\delta}[F_D]$ we deem it appropriate to consider its primitive
\begin{equation}\label{eq:primitive_H}
H_\delta(t) := \int_{[0,t]}h_{\textrm{p},\delta}[F_D](2s)\,ds;
\end{equation}
This function is right-continuous, the limits from the left exist, has jumps at rational times and is continuous elsewhere. We think that $H_\delta$ can be seen as a realization of some stochastic process when $t\in [0,1)$. We do not consider $t>1$ because the derivative of a random process is almost surely non-periodic. We will review briefly some aspects of Lévy processes.

We start defining Poisson point processes, so we have to introduce point functions
\begin{equation*}
p:D_p\subset (0,\infty) \to X,
\end{equation*}
where $D_p$ is countable, and $X$ is some measure space; $X=\R\setminus\{0\}$ in our case. We denote by $\Pi$ the set of all point functions. To every interval $I\subset(0,\infty)$ and measurable set $U\subset X$ we assign the counting function
\begin{equation*}
N_p(I,U) := \abs{\{t\in D_p\cap I\mid p(t)\in U\}}.
\end{equation*} 
We endow $\Pi$ with the minimal $\sigma$-field $\mathcal{B}$ generated by all the functions $p\mapsto N_p(I,U)$.

A (stationary) Poisson point process is a random variable $\vc{p}$ from some probability space $(\Omega,\mathcal{F},P)$ into the space of point functions $(\Pi,\mathcal{B})$, which satisfies, among other properties, 
\begin{equation*}
\mathbf{E}[N_p(I,U)] := \int_\Omega N_{\mathbf p(\omega)}(I,U)\,dP(\omega) = \abs{I}n(U),
\end{equation*}
where $n$ the characteristic measure of the process. We refer the reader to Ch. I.9 of \cite{zbMATH00043057} for details. The point function $p_\delta$ attached to $H_\delta$ represents the location and size of the jumps:
\begin{equation*}
p_\delta: \mathbb{Q}\cap[0,1)\to X = \R\setminus\{0\}.
\end{equation*} 
The following theorem shows that, in a weak sense, $\mathbf{E}[N_p(I,U)]\approx  N_{p_\delta}(I,U)$ for some measure $n$ on $\R\setminus\{0\}$, \textit{i.e.} $p_\delta$ resembles an outcome of some Poisson point process $\vc p$.

\begin{theorem}\label{thm:N_p_delta}
For $I\subset[0,1)$, the function
\begin{equation}\label{eq:thm:def_abs_N}
\abs{N}_{p_\delta}(I,r) := N_{p_\delta}(I,(-\infty,-r]\cup[r,\infty)),\quad \textrm{for }r>0,
\end{equation}
satisfies the bounds
\begin{align}
\abs{N}_{p_\delta}(I,r)&\le C_\delta \abs{I}r^{-1/(1+\delta)} + 1,  & \textrm{all } & r\lesssim_\delta 1, \label{eq:thm:N_p_delta_Upper}\\
\abs{N}_{p_\delta}(I,r) &\gtrsim_\delta \frac{\abs{I}}{\log(c_\delta/r)}r^{-1/(1+\delta)}, & \textrm{all } & r\lesssim_\delta \abs{I}^{2(1+\delta)}.
\end{align}
\end{theorem}

The theorem is consequence of the following lemma.

\begin{lemma} \label{thm:counting_pq}
For $I\subset[0,1)$, the function
\begin{equation*}
M(I,N) := \abs{\{p/q\in I\subset\R \mid q\le N\textrm{ and } (p,q)=1\}},\quad \textrm{for }N\ge 1,
\end{equation*}
satisfies the bounds
\begin{align}
M(I,N)&\le \abs{I}N^2+1, &\textrm{all } & N\ge 1,\label{eq:thm:M_UpperBound}\\
M(I,N)&\gtrsim \abs{I}\frac{N^2}{\log N}, &\textrm{all } & N>2/\abs{I}. \label{eq:thm:M_LowerBound}
\end{align}
\end{lemma}
\begin{proof}
We arrange the rationals inside $I$ in increasing order $p_1/q_1<\cdots<p_M/q_M$ and then use the fact that $p_{i+1}/q_{i+1}-p_i/q_i = 1/(q_{i+1}q_i)$, see Theorem~28 in \cite{zbMATH05309455}, to get
\begin{equation*}
\abs{I}\ge \sum_{i=1}^{M-1}\frac{1}{q_{i+1}q_i}> \frac{M-1}{N^2},
\end{equation*}
which is \eqref{eq:thm:M_UpperBound}.

For the lower bound, we only count fractions $p/q$ with prime denominator. Given a prime $q\le N$ such that $q\abs{I}> 1$, the number of fractions $p/q\in I$ is $\ge q\abs{I}/2$, so for $N> 2/\abs{I}$ we have
\begin{equation*}
M(I,N)\ge \frac{1}{2}\abs{I}\sum_{\abs{I}^{-1}< q\le N} q \ge \frac{1}{4}\abs{I}N\,\abs{\{N/2\le q\le N\mid q \textrm{ prime}\}}.
\end{equation*}
Using the prime number theorem and the Bertrand's postulate we arrive at \eqref{eq:thm:M_LowerBound}.
\end{proof}

We expect the bounds in the lemma can be improved, in particular, the $\log N$-loss in \eqref{eq:thm:M_LowerBound} should be removable. It is interesting to investigate the behavior of $M(I,N)$ when $N\le 2/\abs{I}$. For example, in the interval $I = (0,1/N)$ there is no rational $p/q$ with $q\le N$, so $M(I,N)$ can be zero when $q\le 1/\abs{I}$, but $(0,1/N)$ is a very special interval, can we do any better for other type of intervals?

\begin{proof}[Proof of Theorem~\ref{thm:N_p_delta}]
According to \eqref{eq:thm:h_Talbot} the value of the point function $p_\delta$ at a rational time $t=p/q$ is $p_\delta(t) = a_{\delta,q}/q^{2(1+\delta)}$, where $\abs{a_{\delta,q}}\sim_\delta 1$, so
\begin{equation*}
M(I,c_\delta r^{-\frac{1}{2(1+\delta)}})\le \abs{N}_{p_\delta}(I,r) \le M(I, C_\delta r^{-\frac{1}{2(1+\delta)}}),
\end{equation*}
and the bounds in the theorem follow from Lemma~\ref{thm:counting_pq}.
\end{proof}

Theorem~\ref{thm:N_p_delta} suggests that $N_{p_\delta}(I,U)\approx \abs{I}n(U)$ with characteristic measure $dn(r) \approx r^{-1-1/(1+\delta)}\,dr$. We can write $H_\delta$ in \eqref{eq:primitive_H} in terms of $N_{p_\delta}$ as
\begin{equation*}
H_\delta(t) = \int_{[0,t]}\int_{\R\setminus \{0\}}y\,N_{p_\delta}(dsdy).
\end{equation*}
We recognize here a ``realization'' of an (asymmetric) $\alpha$-Lévy process with exponent $\alpha := 1/(1+\delta)$. We ignore the compensator term because it would add a linear term in $t$, and we can always think of $H_\delta$ as a Lévy process with drift; see Ch. II.3-4 of \cite{zbMATH00043057}. 

This connection between $H_\delta$ and Lévy processes also suggests that $H_\delta$ behaves intermittently, with bursts at rational times with small denominator. It is worth mentioning that $\alpha$-Lévy processes, with $1<\alpha<2$,\footnote{The larger the exponent, the lower the probability of very large jumps.} have already been studied and described as strongly intermittent; see Sec. 3.3 in \cite{zbMATH06420690}.

Yet another evidence of intermittency lies in the variability of the Hölder exponent of $H_\delta$ or multifractality, which is the content of the next theorem, but first we introduce a few definitions and a lemma.
 
\begin{definition}[Hölder exponent]\label{def:Holder}
Let $t_0\in\R$. A function $f$ is in $C^l(t_0)$, for $l\in\R_+$, if there is a polynomial $P_{t_0}$ of degree at most $\lfloor l\rfloor$ such that in a neighborhood of $t_0$
\begin{equation*}
\abs{f(t)-P_{t_0}(t)}\lesssim \abs{t-t_0}^l.
\end{equation*}
The Hölder exponent of $f$ at $t_0$ is
\begin{equation}
\gamma_f(t_0) := \sup\{l\mid f\in C^l(t_0)\}.
\end{equation}
\end{definition} 
 
\begin{definition}[Irrationality measure]\label{def:Irrationality}
Fix $t\in\R$ and let $A\subset\R_+$ be the set of exponents $m\in \R_+$ such that 
\begin{equation*}
0<\Big\lvert t-\frac{p}{q}\Big\rvert <\frac{1}{q^m},
\end{equation*}
has infinitely many solutions. The irrationality measure $\mu(t)$ of $t\in\R$ is
\begin{equation}
\mu(t) := \sup A.
\end{equation}
\end{definition}

If $t$ is rational, then $\mu(t) = 1$; if $t$ is irrational, then by the Dirichlet's approximation theorem $\mu(t)\ge 2$; if $t$ is an irrational algebraic number, then $\mu(t) = 2$ by Roth's theorem; and $t$ is a Liouville number if and only if $\mu(t) = \infty$.

\begin{lemma}\label{thm:Smallest_denominator}
Let $t$ be an irrational number with finite $\mu(t)$ and let $\varepsilon>0$. If $P/Q$ is the fraction with the smallest denominator among all fractions $\abs{t-p/q}<h$, for $h\ll_\varepsilon 1$, then $h^{-1/(\mu+\varepsilon)}<Q$. 

For $t = p_0/q_0$, if $h>0$ and $0<\abs{t-p/q}\le h$, then $q\ge 1/(q_0h)$.
\end{lemma}
\begin{proof}
We only consider the case when $t$ is irrational. Suppose on the contrary that there is $p/q$ with $q\le h^{-1/(\mu+\varepsilon)}$ such that $0<\abs{t-p/q}<h$, then $0<\abs{t-p/q}<1/q^{\mu+\varepsilon}$, but this can only happen for finitely many fractions, so taking $h\ll_\varepsilon 1$ we can avoid those fractions and necessarily $h^{-1/(\mu+\varepsilon)}<Q$.
\end{proof}

\begin{customthm}{\ref{thm:intro:Holder_H}}
{\it
Let $H_\delta$ be the function in \eqref{eq:intro:primitive_H} 
and set $\alpha:= 1/(1+\delta)$.
Then,
\begin{equation} \label{eq:thm:SpecSing2}
d_{H_\delta}(\gamma) = 
\begin{cases}
\alpha\gamma, & \textrm{if }\gamma\in [0,1/\alpha], \\
-\infty, & \textrm{if }\gamma > 1/\alpha.
\end{cases}
\end{equation}
If $t$ is rational, then $\abs{H_\delta(t+h)-H_\delta(t)}\le C_\delta(t)h^{1+2\delta}$ for all $h>0$.
}
\end{customthm}

\begin{proof}
Let $t$ be irrational. 
We first prove that for every $\varepsilon>0$
\begin{equation}\label{eq:Holder_normal}
\abs{H_\delta(t+h)-H_\delta(t)} \le C\abs{h}^{2(1+\delta)/(\mu+\varepsilon)}, \qquad \textrm{for } h\ll_{t,\varepsilon} 1,
\end{equation}
so $\gamma_{H_\delta}(t) \ge 2(1+\delta)/\mu$ if $2\le \mu(t) \le \infty$; 
recall Def.~\ref{def:Holder}.

We assume that $h>0$, the other case being similar. 
We integrate by parts to write the difference as
\begin{align}
H_\delta(t+h)-H_\delta(t) &= \int_{\R\setminus \{0\}} y N_{p_\delta}(I, dy) \notag\\
&= \int_0^\infty [N_{p_\delta}(I, [y,\infty)) - N_{p_\delta}(I, [-y,-\infty))]\,dy.\label{eq:diff_H_raw}
\end{align}
Among all $p/q\in I = (t,t+h]$, let $P/Q$ be the rational 
with the smallest denominator, so
\begin{align*}
H_\delta(t+h)-H_\delta(t) &= \frac{a_{\delta,Q}}{Q^{2(1+\delta)}} + \\
&\hspace*{-8mm}+ \int_0^{\abs{a_{\delta,Q}}/Q^{2(1+\delta)}}[N_{p_\delta}(I, [y,\infty)) - N_{p_\delta}(I, [-y,-\infty))-\frac{a_{\delta,Q}}{\abs{a_{\delta,Q}}}]\,dy \\
&= \frac{a_{\delta,Q}}{Q^{2(1+\delta)}} + J_1.
\end{align*}
To control the integral $J_1$ 
we recall the definition of $\abs{N}_{p_\delta}(I,r)$ in 
\eqref{eq:thm:def_abs_N} and write
\begin{equation*}
\abs{J_1}\le \int_0^{\abs{a_{\delta,Q}}/Q^{2(1+\delta)}}\abs{N}_{p_\delta}(I,r)-1\,dr = \int_0^{\abs{a_{\delta,Q_*}}/Q_*^{2(1+\delta)}}\abs{N}_{p_\delta}(I,r)-1\,dr,
\end{equation*}
where $Q_*>Q$ is the next to the smallest denominator in $I := (t,t+h]$. 
Since $P/Q$ and $P_*/Q_*$ have to be successive in a Farey sequence, 
then $1/Q_*^2<1/(QQ_*)<h$, so, using \eqref{eq:thm:N_p_delta_Upper}, 
we have that $\abs{J_1} \le C_\delta h Q_*^{-2\delta} \le C_\delta h^{1+\delta}$.
Hence,
\begin{equation}\label{eq:diff_H}
H_\delta(t+h)-H_\delta(t) = a_{\delta,Q}/Q^{2(1+\delta)} + \BigO(h^{1+\delta})
\end{equation}
and from Lemma~\ref{thm:Smallest_denominator} we get \eqref{eq:Holder_normal}, so
$\gamma_{H_\delta}(t) \ge 2(1+\delta)/\mu$.

To see that the exponent in \eqref{eq:Holder_normal} is the best possible 
when $2< \mu(t) \le \infty$, let $\{q_i\}_i$ be an infinite list of numbers such that 
$\abs{t-p_i/q_i}<1/q_i^{\mu-\varepsilon}$ for some $\varepsilon>0$. 
If we take $h_i = 1/q_i^{\mu-\varepsilon}$, then 
the smallest denominator in $(t,t+h_i]$ is $Q \le q_i$.
By \eqref{eq:diff_H},  
$\abs{H_\delta(t+h)-H_\delta(t)} \gtrsim h_i^{2(1+\delta)/(\mu-\varepsilon)} $ if $h_i\ll 1$.

Now we show by contradiction that 
for every $\gamma > 1 + \delta$ the set of numbers $t$ 
for which $H_\delta \in C^\gamma(t)$ is empty.
Suppose there is a polynomial 
$P_{t}(h) = H_\delta(t) + hR_{t}(h)$
of degree less than $\floor{\gamma}$ such that
\begin{equation} \label{eq:HolderLarge}
\abs{H_\delta(t + h) - H_\delta(t) - hR_t(h)} \le Ch^\gamma,
\qquad \text{for all } \abs{h} \ll 1.
\end{equation}
By the Dirichlet's approximation theorem
we can choose a sequence $\{q_i\}$ such that
$\abs{t - p_i/q_i} < 1/q_i^2$.
Set $h_i = p_i/q_i - t$ so that $h_i < 1/q_i^2$.
Let us write $H_\delta(s \pm) := \lim_{\varepsilon \to 0^+} H_\delta(s \pm \varepsilon)$,
and notice that
\begin{equation} \label{eq:JumpHdelta}
H_\delta(t + h_i+) - H_\delta(t) = 
	H_\delta(t + h_i-) - H_\delta(t) + \frac{a_{\delta, q_i}}{q_i^{2(1 + \delta)}}.
\end{equation}
From \eqref{eq:HolderLarge} and \eqref{eq:JumpHdelta} we see that
\begin{equation}
\abs{H_\delta(t + h_i+) - H_\delta(t) - h_iR_t(h_i)} \le Ch_i^\gamma,
\end{equation}
and that
\begin{equation}
\abs[\bigg]{H_\delta(t + h_i-) - H_\delta(t) + \frac{a_\delta}{q_i^{2(1 + \delta)}} - h_iR_t(h_i)} \ge 
	\frac{\abs{a_{\delta, q_i}}}{q_i^{2(1 + \delta)}} - 
	Ch_i^\gamma.
\end{equation}
Since $h_i^{1 + \delta} < 1/q_i^{2(1 + \delta)}$ and
$\abs{a_{\delta, q}}\simeq_\delta 1$,
then we conclude that
\begin{equation}
h_i^{1 + \delta} < \frac{1}{q_i^{2(1 + \delta)}} \le C_\delta h_i^\gamma, 
\qquad \text{where } h_i \to 0,
\end{equation}
which is a contradiction when $\gamma > 1 + \delta$,
so $d_{H_\delta}(\gamma) = -\infty$ in this case.

Up to now, we know that for $2 \le \mu(t) \le \infty$ there is a sequence $h_i\to 0$ 
such that $\abs{h_i}^{2(1+\delta)/(\mu-\varepsilon)}\lesssim \abs{H_\delta(t+h_i)-H_\delta(t)}\lesssim \abs{h_i}^{2(1+\delta)/(\mu+\varepsilon)}$, so 
necessarily $\gamma_{H_\delta}(t) = 2(1+\delta)/\mu$ as long as $2(1+\delta)/\mu \neq 1$. 
To prove the theorem we still need to settle the case $2(1+\delta)/\mu = 1$.

From \eqref{eq:diff_H} we see that $H_\delta(t+h)-H_\delta(t) -Ah = a_{\delta,Q}/Q^{2(1+\delta)} -Ah + \BigO(h^{1+\delta})$, 
where $A$ is any constant. 
Again, $\abs{H_\delta(t+h)-H_\delta(t) -Ah} \lesssim \abs{h}^{2(1+\delta)/(\mu+\varepsilon)}$, 
but now, to see that this is best possible, 
we use Dirichlet's theorem to find a sequence $\{q_i\}$ such that $\abs{t-p_i/q_i}<1/q_i^2$. 
We choose $h_i = 1/q_i^2$ and 
notice that $Q = q_i$ is the smallest denominator among all fractions in $(t,t+h_i]$.
Indeed, if $Q < q_i$, 
then we can consider the Farey sequence up to the denominator $q_i$.
Now let $p/q \in (t,t+h_i]$ be contiguous to $p_i/q_i$
and satisfy $q < q_i$ so that
$\abs{p_i/q_i - p/q} = 1/(qq_i) < h_i = 1/q_i^2$,
which implies $q > q_i$, a contradiction.
Thus, the linear term dominates and 
$\abs{H_\delta(t+h_i)-H_\delta(t) -Ah_i} \gtrsim \abs{h_i}$,
which implies $\gamma_{H_\delta}(t) = 1$.

When $t$ is rational $\gamma_{H_\delta}(t) = 0$, but we can still measure 
the Hölder exponent from the right using 
\eqref{eq:diff_H_raw}, \eqref{eq:thm:N_p_delta_Upper} and Lemma~\ref{thm:Smallest_denominator}.

To conclude the theorem we use a result of Güting \cite{zbMATH03174571}, 
which asserts that the Hausdorff dimension of the set of numbers with irrationality $\mu$ is $2/\mu$. 
A short proof can be given using Jarník's theorem \cite[Satz~1]{Jarnik1931}:
\begin{gather}
\mathcal{H}^{2/\mu}(W_\mu) = \infty, \\
W_\mu = \Big\{t\mid \Big|t-\frac{p}{q}\Big|<\frac{1}{q^\mu}\textrm{ for infinitely many fractions } p/q\Big\};
\end{gather}
see also \cite{zbMATH01679547}. 
The set of numbers where $H_\delta$ has Hölder exponent $\gamma \le 1+\delta := 1/\alpha$ 
coincides with the set of numbers with irrationality $2(1+\delta)/\gamma$, and 
the dimension of the latter is $\gamma/(1+\delta)$, 
which is \eqref{eq:thm:SpecSing2}.
\end{proof}

\bibliographystyle{plain}
\bibliography{Sigma}

\end{document}